\documentclass[12pt]{amsart}
\usepackage{fullpage}
\usepackage{amsfonts,amscd}
\usepackage{amssymb}
\usepackage{url}
\usepackage{graphicx}
\usepackage{inputenc}
\usepackage[english]{babel}
\theoremstyle{plain}
\newtheorem{theorem}                {Theorem}      [section]
\newtheorem{proposition}  [theorem]  {Proposition}

\newtheorem{lemma}        [theorem]  {Lemma}

\usepackage[normalem]{ulem}

\theoremstyle{definition}
\newtheorem{example}      [theorem]  {Example}
\newtheorem{remark}       [theorem]  {Remark}

\numberwithin{equation}{section}
\def \t{\mbox{${\mathbb T}$}}
\def \R{{\mathbb R}}
\def \s{\mathbb S}

\def \n{{\mathbb N}}

\usepackage{color}

\def \link {~}
\def \1 {\`}

\DeclareMathOperator{\trace}{trace}

\def \1{\mbox{${\mathbf 1}$}}

\def \n{\mbox{${\mathbb N}$}}

\def \1{\`}

\numberwithin{equation}{section}

\begin{document}

\title[On the second variation of biharmonic tori in spheres]{On the second variation of the biharmonic Clifford torus in $\s^4$}

\author{S.~Montaldo}
\address{Universit\`a degli Studi di Cagliari\\
Dipartimento di Matematica e Informatica\\
Via Ospedale 72\\
09124 Cagliari, Italia}
\email{montaldo@unica.it}

\author{C.~Oniciuc}
\address{Faculty of Mathematics\\ ``Al.I. Cuza'' University of Iasi\\
Bd. Carol I no. 11 \\
700506 Iasi, ROMANIA}
\email{oniciucc@uaic.ro}

\author{A.~Ratto}
\address{Universit\`a degli Studi di Cagliari\\
Dipartimento di Matematica e Informatica\\
Via Ospedale 72\\
09124 Cagliari, Italia}
\email{rattoa@unica.it}

\begin{abstract} The flat torus $\t=\s^1\left (\frac{1}{2} \right ) \times \s^1\left (\frac{1}{2} \right )$ admits a proper biharmonic isometric immersion into the unit $4$-dimensional sphere $\s^4$ given by $\Phi=i \circ \varphi$, where $\varphi:\t \to \s^3(\frac{1}{\sqrt 2})$ is the minimal Clifford torus and $i:\s^3(\frac{1}{\sqrt 2}) \to \s^4$ is the biharmonic small hypersphere. 
\noindent The first goal of this paper is to compute the biharmonic \textit{index} and \textit{nullity} of the proper biharmonic immersion $\Phi$.
\noindent After, we shall study in the detail the kernel of the generalised Jacobi operator $I_2^\Phi$. We shall prove that it contains a direction which admits a natural variation with vanishing first, second and third derivatives, and such that the fourth derivative is negative.
\noindent In the second part of the paper we shall analyse the specific contribution  of $\varphi$ to the biharmonic index and nullity of $\Phi$. In this context, we shall study a more general composition $\tilde{\Phi}=\tilde{\varphi} \circ i$, where $\tilde{\varphi}: M^m \to \s^{n-1}(\frac{1}{\sqrt 2})$, $ m \geq 1$, $n \geq {3}$, is a minimal immersion and $i:\s^{n-1}(\frac{1}{\sqrt 2}) \to \s^n$ is the biharmonic small hypersphere. First, we shall determine a general sufficient condition which ensures that the second variation of $\tilde{\Phi}$ is nonnegatively defined on $\mathcal{C}\big (\tilde{\varphi}^{-1}T\s^{n-1}\big )$. Then we complete this type of analysis on our Clifford torus and, as a complementary result, we obtain the $p$-harmonic index and nullity of $\varphi$. In the final section we compare our general results with those which can be deduced from the study of the \textit{equivariant second variation}.
\end{abstract}



\thanks{The first and the last author were supported by Fondazione di Sardegna (project STAGE) and are members of the Italian National Group G.N.S.A.G.A. of INdAM. The second author was partially supported by the grant PN-III-P4-ID-PCE-2020-0794.}
\maketitle
\section{Introduction}\label{intro}
{\it Harmonic maps} are the critical points of the {\em energy functional}
\begin{equation}\label{energia}
E(\phi)=\frac{1}{2}\int_{M}\,|d\phi|^2\,dv_M \,\, ,
\end{equation}
where $\phi:M\to N$ is a smooth map from a compact Riemannian
manifold $(M^m,g)$ to a Riemannian
manifold $(N^n,h)$. In particular, $\phi$ is harmonic if it is a solution of the Euler-Lagrange system of equations associated to \eqref{energia}, i.e.,
\begin{equation}\label{harmonicityequation}
  - d^* d \phi =   {\trace} \, \nabla d \phi =0 \,\, .
\end{equation}
The left member of \eqref{harmonicityequation} is a vector field along the map $\phi$ or, equivalently, a section of the pull-back bundle $\phi^{-1} TN$: it is called {\em tension field} and denoted $\tau (\phi)$. In addition, we recall that, if $\phi$ is an \textit{isometric immersion}, then $\phi$ is a harmonic map if and only if it defines a \textit{minimal submanifold} of $N$ (see \cite{EL1, EL83} for background).

A related topic of growing interest is the study of {\it biharmonic maps}. As suggested in \cite{EL83}, \cite{ES}, these maps, which provide a natural generalisation of harmonic maps, are defined as the critical points of the {\it bienergy functional} 
\begin{equation*}\label{bienergia}
    E_2(\phi)=\frac{1}{2}\int_{M}\,|d^*d\phi|^2\,dv_M=\frac{1}{2}\int_{M}\,|\tau(\phi)|^2\,dv_M\,\, .
\end{equation*}
There have been extensive studies on biharmonic maps. We refer to \cite{Chen, Jiang, LO, SMCO, Ou, ChenOu} for an introduction to this topic.
We observe that, obviously, any harmonic map is trivially biharmonic and an absolute minimum for the bienergy. Therefore, we say that a biharmonic map is {\it proper} if it is not harmonic and, similarly, a biharmonic isometric immersion is {\it proper} if it is not minimal.

Our paper is devoted to the study of the second variation of the bienergy functional. In order to introduce this topic, it is convenient to recall some basic facts about the generalised Jacobi operator $I_2^\phi(V)$ and the definition of index and nullity. More specifically, let $\phi:M\to N$ be a biharmonic map. We shall consider a two-parameter smooth variation $\left \{ \phi_{t,s} \right \}$ $(-\varepsilon <t,s < \varepsilon,\,\phi_{0,0}=\phi)$ and denote by $V,W$ their associated vector fields:
\begin{align}\label{V-W}
& V(x)= \left . \frac{\partial}{\partial t}\right |_{t=0} \phi_{t,0} (x) \in T_{\phi(x)}N \\ \nonumber
& W(x)= \left . \frac{\partial}{\partial s}\right |_{s=0} \phi_{0,s}(x) \in T_{\phi(x)}N\,.
\end{align}
Note that $V$ and $W$ are sections of $\phi^{-1}TN$. The {\it Hessian} of the bienergy functional $E_2$ at its critical point $\phi$ is defined by
\begin{equation*}\label{Hessian-definition}
H(E_2)_\phi (V,W)= \left . \frac{\partial^2}{\partial t \partial s}\right |_{(t,s)=(0,0)}  E_2 (\phi_{t,s}) \, .
\end{equation*}
The following theorem was obtained by Jiang and translated by Urakawa \cite{Jiang}:
\begin{theorem}\label{Hessian-Theorem} Let $\phi:M\to N$ be a biharmonic map between two Riemannian manifolds $(M^m,g)$ and $(N^n,h)$, where $M$ is compact. Then the Hessian of the bienergy functional $E_2$ at a critical point $\phi$ is given by
\begin{equation*}\label{Operator-Ir}
H(E_2)_\phi (V,W)=\left (  I_2^\phi(V),W \right )= \int_M \langle I_2^\phi(V),W \rangle \, dv_M \,\,,
\end{equation*}
where $I_2^\phi=\,:\mathcal{C}\left(\phi^{-1} TN\right) \to \mathcal{C}\left(\phi^{-1} TN\right)$ is a semilinear elliptic operator of order $4$.
\end{theorem}
When the context is clear, we shall just write $I_2$ instead of $I_2^\phi$. 

Next, we want to give an explicit description of the operator $I_2$. To this purpose, let $\nabla^M, \nabla^N$ and $\nabla^{\phi}$ be the induced connections on the bundles $TM, TN$ and $\phi ^{-1}TN$ respectively. Then the \textit{rough Laplacian} on sections of $\phi^{-1} TN$, denoted by $\overline{\Delta}$, is defined by
\begin{equation*}\label{roughlaplacian}
    \overline{\Delta}=d^* d =-\sum_{i=1}^m\Big\{\nabla^{\phi}_{e_i}
    \nabla^{\phi}_{e_i}-\nabla^{\phi}_
    {\nabla^M_{e_i}e_i}\Big\}\,\,,
\end{equation*}
where $\{e_i\}_{i=1}^m$ is a local orthonormal frame field tangent to $M$. 

Let $\s^n(R)$ denote the Euclidean $n$-dimensional sphere of radius $R$. We shall write $\s^n$ when $R=1$. In the present paper, we shall only need the explicit expression of $I_2(V)$ in the case that the target manifold is $\s^n$. This useful formula, which was first given in \cite{Onic} and can be deduced from a more general result in \cite{Jiang}, is the following:
\begin{eqnarray}\label{I2-general-case}
\nonumber
I_2(V)&=&\overline{\Delta}^2 V+\overline{\Delta}\left ( {\rm trace}\langle V,d\phi \cdot \rangle d\phi \cdot - |d\phi|^2\,V \right)+2\langle d\tau(\phi),d\phi \rangle V+|\tau(\phi)|^2 V\\ \nonumber
&&-2\,{\rm trace}\langle V,d\tau(\phi) \cdot \rangle d\phi \cdot-2 \,{\rm trace} \langle \tau(\phi),dV \cdot \rangle d\phi \cdot - \langle \tau(\phi),V \rangle \tau(\phi)\\ 
&&+{\rm trace} \langle d\phi \cdot,\overline{\Delta}V \rangle d\phi \cdot +{\rm trace}\langle d\phi \cdot,\left ( {\rm trace}\langle
V,d\phi \cdot \rangle d\phi \cdot \right ) \rangle d\phi \cdot \\ \nonumber &&-2 |d\phi|^2\, {\rm trace}\langle d\phi \cdot,V \rangle
 d\phi\cdot 
+2 \langle dV,d\phi\rangle \tau(\phi) -|d\phi|^2 \,\overline{\Delta}V+|d\phi|^4 V\,, \nonumber
\end{eqnarray}
where $\cdot$ denotes trace with respect to a local orthonormal frame field on $M$. 

Next, we recall from the general theory that, since $M$ is compact, the spectrum
\begin{equation*}
\mu_1 < \mu_2 < \ldots < \mu_i < \ldots
\end{equation*}
of the generalised Jacobi operator $I_2(V)$ associated to the bienergy is \textit{discrete} and tends to $+\infty$ as $i$ tends to $+ \infty$. We denote by $\mathcal{V}_i$ the eigenspace associated to the eigenvalue $\mu_i$. Then we can define the biharmonic \textit{index} of $\phi$ as follows (see, for instance, \cite{Urakawa}):
\begin{equation}\label{Index-definition}
{\rm Index}_2(\phi) = \sum_{\mu_i <0} \dim(\mathcal{V}_i)\,.
\end{equation}
The biharmonic \textit{nullity} of $\phi$ is defined as the dimension of the kernel of $I_2$:
\begin{equation}\label{Nullity-definition}
{\rm Null}_2(\phi) =  \dim \left \{ V \in\mathcal{C}\left(\phi^{-1} TN\right) \, : \, I_2(V)=0\right \} =\dim \left({\rm Ker}(I_2)\right )\,.
\end{equation}
We say that a biharmonic map $\phi: M\to N$ is {\it stable} if ${\rm Index}_2(\phi)=0$. As pointed out in \cite{EL83}, this definition has to be regarded as a notion of \textit{second order stability}. This notion has a geometric meaning, that is it does not depend on the variation $\{\phi_t\}$ but only on the associated vector field $V=\partial \phi_t \slash \partial t \big|_{t=0}$. On the other hand, we know that for a variation $\{\phi_t\}$ with associated vector field $V$  which belongs to ${\rm Ker}(I_2)$, it may happen that 
\begin{equation}\label{high-order-inst}
\left. \frac{d^\ell}{dt^\ell} \right|_{t=0} E_2 (\phi_t)=0\,, \quad \ell=1,\ldots,k-1 \,; \quad
\left. \frac{d^k}{dt^k}\right|_{t=0} E_2 (\phi_t)  <0 \,\quad {\rm for\, some\, even }\,\,k\geq 4\,.
\end{equation}
If \eqref{high-order-inst} occurs, then the function $E_2 (\phi_t)$ has a local maximum at $t=0$.

In general, the study of the second variation of the bienergy functional is a complicated task and there are not many papers dealing with computations and estimates of the index and nullity of some proper biharmonic submanifolds in the Euclidean unit sphere $\s^n$ (for instance, see \cite {BFO,LO, TAMS, MOR3}). 

A natural first step is to investigate the second variation of a proper biharmonic submanifold of $\s^{n}$ which lies in the small hypersphere $\s^{n-1}(\frac{1}{\sqrt 2})$ as a minimal submanifold. In this order of ideas, one of the simpler examples is the \textit{proper biharmonic Clifford torus} in $\s^4$. More precisely, let $\t$ be the flat torus with radii equal to $1\slash 2$, i.e., 
\begin{equation}\label{def-toro}
\t=\s^1\left (\frac{1}{2} \right ) \times \s^1\left (\frac{1}{2} \right )\,.
\end{equation}
and denote by $\Phi:\t \to \s^4$ the proper biharmonic isometric immersion obtained as the composition of the minimal Clifford torus $\varphi:\t \to \s^3 \left ( \frac{1}{\sqrt 2}\right ) $ followed by the proper biharmonic inclusion $
i:\s^3 \left ( \frac{1}{\sqrt 2}\right ) \to \s^4$.
This example was partially discussed in \cite{LO}, where the authors conjectured that its biharmonic index is equal to $1$. In this paper we shall complete the study of this example and compute the exact values of both Null$_2(\Phi)$ and Index$_2(\Phi)$.

For future use, we point out that the we shall perform all the computations related to $\Phi=i \circ \varphi$ using the following explicit description:
\begin{eqnarray}\label{eq-mapdatoroasfera} \nonumber
\Phi:\t\quad &\to &\s^4 \subset \R^5\\
    \left ( \gamma, \vartheta \right )  &\mapsto &\left (\frac{1}{2} \cos \gamma,  \frac{1}{2} \sin \gamma,\frac{1}{2} \cos \vartheta,  \frac{1}{2} \sin \vartheta, \frac{1}{\sqrt 2}  \right ) \,, \quad 0 \leq \gamma, \vartheta \leq 2 \pi \,.
\end{eqnarray}

The first result of our paper is:
\begin{theorem}\label{Th-Main}
Let $\Phi:\t \to \s^4 $ be the proper biharmonic immersion defined in \eqref{eq-mapdatoroasfera}. Then
\begin{equation*}\label{eq-index+nullity-main}
\begin{array}{llll}
{\rm (i)}&{\rm Index}_2(\Phi)&=& 1 \,; \\ 
{\rm (ii)}&{\rm Null}_2(\Phi)&=& 11 \,. 
\end{array}
\end{equation*}
\end{theorem}
The proof of Theorem\link\ref{Th-Main} will be carried out in Section\link\ref{proofs}.

In Section\link\ref{section-null} we shall completely determine the structure of ${\rm Ker}(I_2)$ putting in evidence the existence of sections which are not originated from the Killing vector fields. We shall also show that the interesting phenomenon \eqref{high-order-inst} does occur for a suitable variation and  $k=4$.

We observe that the Clifford torus $\varphi:\t \to \s^3 \left ( \frac{1}{\sqrt 2}\right ) $ is minimal, and so it is a stable critical point for the bienergy. By contrast, the small hypersphere $
i:\s^3 \left ( \frac{1}{\sqrt 2}\right ) \to \s^4$ is unstable with biharmonic index equal to $1$ (see \cite{LO}). Then it seems of interest to study the effects of this composition on the second variation. This type of analysis will be carried out in Section\link\ref{Composition}, where we shall study the specific contribution of the minimal immersion $\varphi$ to the biharmonic index of the composition $\Phi=i \circ \varphi$.  

In this context, we shall study a more general composition $\tilde{\Phi}=\tilde{\varphi} \circ i$, where $\tilde{\varphi}: M^m \to \s^{n-1}(\frac{1}{\sqrt 2})$, $ m \geq 1$, $n \geq {3}$, is a minimal immersion and $i:\s^{n-1}(\frac{1}{\sqrt 2}) \to \s^n$ is the biharmonic small hypersphere. First, we shall determine a general sufficient condition which ensures that the second variation of $\tilde{\Phi}$ is nonnegatively defined on $\mathcal{C}\big (\tilde{\varphi}^{-1}T\s^{n-1}\big )$. Then we complete this type of analysis on our Clifford torus.

As a complementary result, we shall also obtain the \textit{$p$-harmonic} index and nullity of the harmonic map $\varphi$ for any $p\geq 2$, where a \textit{$p$-harmonic map} $\phi:M\to N$ is a critical point of the \textit{$p$-energy functional}
\begin{equation}\label{p-energy-def}
E_{(p)}(\phi)=\frac{1}{p}\int_{M}\,|d\phi|^p\,dv_M \,.
\end{equation}

In general, working on the whole  $C^{\infty}(M,N)$, it is very difficult to carry out a complete study of the second variation of a given biharmonic map. Therefore, in order to obtain some partial, but geometrically interesting results, it seems of interest to develop a method of investigation reducing in a suitable way the domain of the bienergy functional. More precisely, when a compact Lie group of isometries $G$ acts on both $M$ and $N$, we can restrict to the set $\Sigma$ of all symmetric points with respect to the natural action of $G$ on $C^{\infty}(M,N)$.
In this spirit, in the final section, because the Clifford torus $\Phi(\t)$ in $\s^4$ is a $G$-invariant submanifold with $G={\rm SO}(2) \times {\rm SO}(2)$, we compare our general results with those which can be deduced from the study of the second variation restricted to the set $\Sigma$ of all symmetric points of $C^{\infty}(\t,\s^4)$ with respect to the action of $G$ (see \cite{Palais}).
\section{Proof of Theorem \ref{Th-Main}}\label{proofs}
The first step is to derive an explicit expression for the operator $I_2:\mathcal{C}\left(\Phi^{-1} T\s^4\right) \to \mathcal{C}\left(\Phi^{-1} T\s^4\right)$ using formula \eqref{I2-general-case}. To this purpose, we first introduce suitable vector fields along $\Phi$. 

More specifically, using Cartesian coordinates $y=\left(y^1,y^2,y^3,y^4,y^5 \right)$ on $\R^5$, we define
\begin{equation}\label{eq-def-V-vari}
V_{\gamma}=Y_{\gamma}(\Phi) \,; \quad V_{\vartheta}=Y_{\vartheta}(\Phi) \,; \quad V_{\nu}=Y_{\nu}(\Phi) \,; \quad V_{\eta}=Y_{\eta}(\Phi) \,,
\end{equation}
where
\begin{equation}\label{eq-definizioneY-vari}
\begin{array}{lll}
Y_{\gamma}&=& 2 \left (-y^2,y^1,0,0,0 \right ) \left ( =\left( -2 y^2 \frac{\partial}{\partial y^1}+2y^1 \frac{\partial}{\partial y^2}\right )\right ) \,; \\ 
Y_{\vartheta}&=& 2 \left (0,0,-y^4,y^3,0 \right ) \,; \\
Y_{\nu}&=& \sqrt 2 \left (y^1,y^2,-y^3,-y^4,0 \right ) \,; \\
Y_{\eta}&=&  \left (y^1,y^2,y^3,y^4,-\frac{1}{\sqrt 2} \right ) \,.  
\end{array}
\end{equation}
From a geometric viewpoint, we observe that $\left \{V_{\gamma},V_{\vartheta}, V_{\nu},V_{\eta} \right \}$ provides an orthonormal basis of $T\s^4$ at each point of the image of $\Phi$. Moreover, we point out that $V_{\gamma},V_{\vartheta}$ are tangent to the Clifford torus $\Phi \left(\t \right)$, $V_{\nu}$ represents the normal direction to the torus in $\s^3 \left( \frac{1}{\sqrt 2}\right )$, while $V_{\eta}$ is the normal to $\s^3 \left( \frac{1}{\sqrt 2}\right )$ in $\s^4$.

By way of summary, we conclude that each section $V \in \mathcal{C}\left( \Phi^{-1}T\s^4\right )$ can be written as
\begin{equation}\label{eq-general-section-toro-sfera}
V=f_1 \,V_\gamma +f_2\,V_\vartheta+f_3 \,V_\nu +f_4\,V_\eta \,,
\end{equation}
where $f_j\in C^\infty \left ( \t \right)$, $j=1,\ldots,4$. We also point out that $X_\gamma$ and $ X_\vartheta$, where
\begin{equation*}\label{eq-orth-base-dominio}
X_\gamma = 2 \,\frac{\partial}{\partial \gamma}\,,\, \quad X_\vartheta = 2 \,\frac{\partial}{\partial \vartheta}\,,
\end{equation*}
are globally defined vector fields which form an orthonormal basis at each point of $T\t$. Moreover, $d \Phi(X_\gamma)=Y_{\gamma}(\Phi)=V_\gamma$, $d \Phi(X_\vartheta)=Y_{\vartheta}(\Phi)=V_\vartheta$. 

For our purposes, it shall be sufficient to study in detail the case that the functions $f_j$ in \eqref{eq-general-section-toro-sfera} are eigenfunctions of the Laplacian.

Our first goal is to compute $\overline{\Delta} V_\gamma$, $\overline{\Delta} V_\vartheta$, $\overline{\Delta} V_\nu$, $\overline{\Delta} V_\eta$. To this purpose, we observe that the choice of vector fields of type \eqref{eq-definizioneY-vari} suggests that the simplest way to compute covariant derivatives along $\Phi$ is to use the following well-known formula, where $W$ is a section of the pull-back bundle $\Phi^{-1}T\s^4$ given by the composition of a (local) vector field on $\s^4$, again denoted by $W$, with $\Phi$:
\begin{equation*}\label{eq-cov-deriv-formula}
\nabla_X^\Phi W =  \nabla^{{\mathbb S}^4}_{d\Phi(X)} W =\nabla^{\R^5}_{d\Phi(X)} W+ \langle d\Phi(X),W \rangle \,\Phi \,.
\end{equation*}
We have:
\begin{eqnarray*}\label{eq-nabla-V-theta}\nonumber
\nabla_{X_\gamma}^\Phi V_\gamma &=& \nabla_{X_\gamma}^\Phi Y_\gamma(\Phi)=\nabla_{Y_\gamma(\Phi)}^{\s^4} Y_\gamma  \\\nonumber 
&=&  \nabla_{Y_\gamma(\Phi)}^{\R^5} Y_\gamma + \langle Y_\gamma (\Phi),Y_\gamma (\Phi) \rangle \,\Phi \\
&=& \left . \left [ \nabla^{\R^5}_{-2 y^2 \partial \slash \partial y^1+2 y^1 \partial \slash \partial y^2 }\left (-2\,y^2\,\frac{\partial}{\partial y^1}+2\,y^1\,\frac{\partial}{\partial y^2}\right )+ \left ( y^1,y^2,y^3,y^4,y^5\right )\right ] \right |_{\Phi}\\\nonumber
&=&\left . \left [\left ( -4 y^1,-4 y^2,0,0,0\right )+\left ( y^1,y^2,y^3,y^4, y^5\right )\right ]\right |_{\Phi}\\\nonumber
&=&\left .\left ( -3 y^1,-3 y^2,y^3,y^4, y^5\right )\right |_{\Phi}\\\nonumber
&=&- \sqrt 2 \, V_{\nu} - V_{\eta}\,.
\end{eqnarray*}
Similarly, using the same method of computation we obtain the following identities:

\begin{equation}\label{eq-nabla-V-tutti}
\begin{array}{lcl}
\nabla_{X_\gamma}^\Phi V_\gamma =-\, \sqrt 2 \, V_{\nu} - V_{\eta}\,; &\quad&\nabla_{X_\gamma}^\Phi V_\nu = \sqrt 2 \, V_\gamma\,; \\
\nabla_{X_\vartheta}^\Phi V_\gamma = 0\,; &\quad&\nabla_{X_\vartheta}^\Phi V_\nu =-\,\sqrt 2 \, V_\vartheta\,; \\
\nabla_{X_\gamma}^\Phi V_\vartheta = 0\,; &\quad& \nabla_{X_\gamma}^\Phi V_\eta = V_\gamma\,;\\
\nabla_{X_\vartheta}^\Phi V_\vartheta = \sqrt 2 \, V_{\nu} - V_{\eta}\,; &\quad& \nabla_{X_\vartheta}^\Phi V_\eta = V_\vartheta\,.
\end{array}
\end{equation}
Next, we compute using \eqref{eq-nabla-V-tutti} and obtain:
\begin{equation}\label{eq-Deltabar-V-tutti}
\begin{array}{llllll}
\overline{\Delta} V_\gamma &=&3\,V_\gamma \,; \qquad &\overline{\Delta} V_\vartheta &=& 3\,V_\vartheta \,;\\
\overline{\Delta} V_\nu &=&  4\,V_\nu \,; \qquad &\overline{\Delta} V_\eta &=& 2\,V_\eta\,.
\end{array}
\end{equation}
By way of example, we detail here the steps to obtain the first equality in \eqref{eq-Deltabar-V-tutti}:
\begin{eqnarray*}
\overline{\Delta} V_\gamma&=&- \left [\nabla^\Phi_{X_\gamma}\,\nabla^\Phi_{X_\gamma} V_\gamma +\nabla^\Phi_{X_\vartheta}\,\nabla^\Phi_{X_\vartheta} V_\gamma \right] \\
&=&- \left [\nabla^\Phi_{X_\gamma} \left (-\,\sqrt 2 \, V_\nu -\,V_\eta \right) +0\right]\\
&=& \sqrt 2 \, \nabla^\Phi_{X_\gamma} V_\nu +\nabla^\Phi_{X_\gamma} V_\eta \\
&=& 2\, V_\gamma + V_\gamma= 3\, V_\gamma \,. 
\end{eqnarray*}
In the sequel, we shall denote by $\Delta$ the Laplace operator on $\t$, i.e.,
\begin{equation}\label{eq:definition-laplace}
\Delta= -4\left (\frac{\partial^2}{\partial \gamma^2}+ \frac{\partial^2}{\partial \vartheta^2}\right )\,.
\end{equation}
\begin{lemma}\label{lemma3} Assume that $\Delta f= \lambda f$. Then
\begin{eqnarray*}\label{rough-fV-tutti}
{\rm (i)}\,\,\quad \overline{\Delta} (fV_\gamma)&=&(\lambda+3)f\, V_\gamma+4 \sqrt 2 \,f_\gamma \,V_\nu+4 f_\gamma\,V_\eta\,;\\ \nonumber
{\rm (ii)}\,\, \quad \overline{\Delta} (f V_\vartheta)&=&
(\lambda+3)f \,V_\vartheta-4 \sqrt 2 \,f_\vartheta \,V_\nu+4 f_\vartheta\,V_\eta\,;
\\\nonumber
{\rm (iii)}\,\,\quad \overline{\Delta} (fV_\nu)&=&-4 \sqrt 2 \,f_\gamma \,V_\gamma+4 \sqrt 2 f_\vartheta\,V_\vartheta+(\lambda+4)f \,V_\nu\,; \\ \nonumber
{\rm (iv)}\,\, \quad \overline{\Delta} (f V_\eta)&=&
-4  f_\gamma \,V_\gamma-4  f_\vartheta\,V_\vartheta+(\lambda+2)f \,V_\eta\,. \nonumber
\end{eqnarray*}
\end{lemma}
\begin{proof}[Proof of Lemma~\ref{lemma3}]
This lemma can be easily proved using \eqref{eq-nabla-V-tutti} and \eqref{eq-Deltabar-V-tutti} together with the general formula
\begin{equation*}\label{general-product-formula}
\overline{\Delta} (f\,V)=(\Delta f)\,V-2\,\nabla_{\nabla f}^{\Phi}V+f\,\overline{\Delta} V \,,
\end{equation*}
where $\nabla f= 2f_\gamma X_\gamma +2   f_\vartheta X_\vartheta$.
\end{proof}
\begin{lemma}\label{lemma4} Assume that $\Delta f= \lambda f$. Then
\begin{eqnarray*}\label{rough2-fV-tutti}
{\rm (i)}\,\,\quad \overline{\Delta}^2 (fV_\gamma)&=&\left[ (\lambda+3)^2 f -48 f_{\gamma \gamma}\right ]V_\gamma + \left [ 16f_{\gamma_\vartheta}\right ]V_\vartheta\\\nonumber
&&+ \left [ 4 \sqrt 2 (2\lambda+7)f_\gamma \right ]V_\nu +\left [ 4  (2\lambda+5)f_\gamma \right ]V_\eta\,; \\\nonumber
{\rm (ii)}\,\, \quad \overline{\Delta}^2 (f V_\vartheta)&=&\left [ 16f_{\gamma_\vartheta}\right ]V_\gamma + \left[ (\lambda+3)^2 f -48 f_{\vartheta \vartheta}\right ]V_\vartheta \\\nonumber
&&- \left [ 4 \sqrt 2 (2\lambda+7)f_\vartheta \right ]V_\nu +\left [ 4  (2\lambda+5)f_\vartheta \right ]V_\eta\,; \\\nonumber
{\rm (iii)}\,\,\quad \overline{\Delta}^2 (fV_\nu)&=&-\left[4 \sqrt 2 (2\lambda+7)f_\gamma \right ]V_\gamma +\left[4 \sqrt 2 (2\lambda+7)f_\vartheta \right ] V_\vartheta\\\nonumber
&&+ \left [ (\lambda^2 +16 \lambda +16) f \right ]V_\nu +\left [ -16 \sqrt 2 f_{\gamma \gamma}+16 \sqrt 2 f_{\vartheta \vartheta} \right ]V_\eta\,; \\ \nonumber
{\rm (iv)}\,\, \quad \overline{\Delta}^2 (f V_\eta)&=&-\left[ 4(2\lambda+5) f_\gamma \right ]V_\gamma -\left[ 4(2\lambda+5) f_\vartheta \right ]V_\vartheta\\\nonumber
&&+\left [ -16 \sqrt 2 f_{\gamma \gamma}+16 \sqrt 2 f_{\vartheta \vartheta} \right ]V_\nu +\left [ (\lambda^2+8\lambda+4)f \right ]V_\eta\,.\nonumber
\end{eqnarray*}
\end{lemma}
\begin{proof}[Proof of Lemma~\ref{lemma4}] Since $X_\gamma$ is a Killing field on $\t$, $\Delta f_\gamma= \lambda\,f_\gamma$. Similarly, $\Delta f_\vartheta= \lambda\,f_\vartheta$. Then the lemma can be proved using Lemma\link\ref{lemma3} and again \eqref{eq-nabla-V-tutti}, \eqref{eq-Deltabar-V-tutti} together with 
\eqref{eq:definition-laplace}.
\end{proof}
Our first key result is:
\begin{proposition}\label{proposizione-I-esplicito-toro-sfera} Assume that $f\in C^{\infty}\left ( \t \right )$ is an eigenfunction of $\Delta$ with eigenvalue $\lambda$. Then
\begin{equation}\label{eq-I2-fV-tutti}
\begin{array}{llll}
{\rm (i)}&I_2 \left ( f\, V_\gamma \right )&=&\left[\lambda(4+\lambda)f-48 f_{\gamma \gamma} \right ]V_\gamma + 16 f_{\gamma \vartheta} V_\vartheta   \\
&&&+ \left [8 \sqrt 2 (2+\lambda) f_\gamma \right ]V_\nu+ 8 \lambda f_\gamma \,V_\eta \,;\\
&&&\\
{\rm (ii)}&I_2 \left ( f\, V_\vartheta \right )&=& 16 f_{\gamma \vartheta}V_\gamma+ \left[\lambda(4+\lambda)f-48 f_{\vartheta \vartheta} \right ] V_\vartheta   \\
&&&- \left [8 \sqrt 2 (2+\lambda) f_\vartheta \right ]V_\nu+ 8 \lambda f_\vartheta \,V_\eta \,;\\
&&&\\
{\rm (iii)}&I_2 \left ( f\, V_\nu \right )&=& -\left [8 \sqrt 2 (2+\lambda)f_\gamma \right ]V_\gamma+\left [8 \sqrt 2 (2+\lambda)f_\vartheta \right ]V_\vartheta \\
&&&+\left [\lambda (12+\lambda)f \right ]V_\nu+\left [-16 \sqrt 2 f_{\gamma \gamma}+16 \sqrt 2 f_{\vartheta \vartheta} \right ]V_\eta\,;\\
&&&\\
{\rm (iv)}&I_2 \left ( f\, V_\eta \right )&=&   -8 \lambda f_\gamma \,V_\gamma-8 \lambda f_\vartheta \,V_\vartheta \\
&&&+\left [-16 \sqrt 2 f_{\gamma \gamma}+16 \sqrt 2 f_{\vartheta \vartheta} \right ]V_\nu+
\left [(\lambda^2+4\lambda-16)f \right ]V_\eta\,.\\
\end{array}
\end{equation}
\end{proposition}
\begin{proof} The proof of Proposition\link\ref{proposizione-I-esplicito-toro-sfera} amounts to computing all the $13$ terms which appear in the right-hand side of formula \eqref{I2-general-case} and adding them up. 

In this proposition we only have to deal with sections of $\Phi^{-1}T\s^4$ which are of the type $V=f V^*$, where $f$ is an eigenfunction of $\Delta$ on the torus and $V^*$ is one amongst the $4$ vector fields defined in \eqref{eq-def-V-vari}. 

We use Roman numbers to denote the $13$ terms in \eqref{I2-general-case} and now we show how to compute each of them.

\textbf{Term I.} It is of the type $\overline{\Delta}^2 (fV^*)$ and was computed in Lemma\link\ref{lemma4}.

\textbf{Term II.} Since $\Phi$ is an isometric immersion from a $2$-dimensional domain, we have $|d\Phi|^2=2$. Therefore,
\[
\overline{\Delta}\left ( {\rm trace}\langle (fV^*),d\Phi \cdot \rangle d\Phi \cdot - |d\Phi|^2\,(fV^*) \right)=\overline{\Delta}\left (f\langle V^*,V_\gamma\rangle V_\gamma\right)+\overline{\Delta}\left (f \langle V^*,V_\vartheta\rangle V_\vartheta\right)-2\overline{\Delta} (f V^* ) \,.
\]
Since each of the two scalar products $\langle V^*,V_\gamma\rangle$, $\langle V^*,V_\vartheta\rangle$ is either equal to $0$ or to $1$, we conclude readily that term II can now be computed using directly Lemma\link\ref{lemma3}.

\textbf{Term III.} First, we observe that, since $\Phi=i \circ \varphi$
\begin{equation*}\label{eq-tau-Phi}
\tau(\Phi)=d i(\tau (\varphi))+ {\rm trace}\nabla d i (d \varphi \cdot,d \varphi \cdot)= - 2 V_\eta \,,
\end{equation*}
where we used the fact that $\varphi$ is minimal and the second fundamental form of the small hypersphere $\s^3(1\slash \sqrt 2)$ in $\s^4$ is $B(X,Y)= -\langle X,Y \rangle \eta$, with $\eta|_{\Phi}=V_\eta$.

Then
\[
2\langle d\tau(\Phi),d\Phi \rangle (fV^*)=-4 \langle \nabla^{\Phi}_{X_\gamma} V_\eta,V_\gamma \rangle (fV^*)-4\langle \nabla^{\Phi}_{X_\vartheta} V_\eta,V_\vartheta \rangle (fV^*)=-8 f V^*\,,
\]
where the last equality is an immediate consequence of \eqref{eq-nabla-V-tutti}. 

\textbf{Term IV.} $|\tau(\Phi)|^2 (fV^*)=4 (fV^*)$.

\textbf{Term V.}
\begin{eqnarray*}
-2\,{\rm trace}\langle V,d\tau(\Phi) \cdot \rangle d\Phi \cdot&=&4f\langle V^*,\nabla^{\Phi}_{X_\gamma} V_\eta \rangle V_\gamma+4f\langle V^*,\nabla^{\Phi}_{X_\vartheta} V_\eta \rangle V_\vartheta\\
&=&4f\langle V^*, V_\gamma \rangle V_\gamma+4f\langle V^*,V_\vartheta \rangle V_\vartheta\,.
\end{eqnarray*}
\textbf{Term VI.}
\[
-2 \,{\rm trace} \langle \tau(\Phi),d(fV^*) \cdot \rangle d\Phi \cdot =4  \langle V_\eta,2 f_\gamma V^*+f \nabla^{\Phi}_{X_\gamma} V^* \rangle V_\gamma+ 
4  \langle V_\eta,2 f_\vartheta V^*+f \nabla^{\Phi}_{X_\vartheta} V^* \rangle V_\vartheta\,.
\]
Also this computation can now be ended easily using \eqref{eq-nabla-V-tutti}.

\textbf{Term VII.} $- \langle \tau(\Phi),(fV^*) \rangle \tau(\Phi)=-4 f\langle V_\eta,V^* \rangle V_\eta \,.$

\textbf{Term VIII.}
\[
{\rm trace} \langle d\Phi \cdot,\overline{\Delta}(fV^*) \rangle d\Phi \cdot =
\langle V_\gamma,\overline{\Delta}(fV^*) \rangle V_\gamma +\langle V_\vartheta,\overline{\Delta}(fV^*) \rangle V_\vartheta \,.
\]
So this term can be calculated using Lemma\link\ref{lemma3}.

\textbf{Term IX.} Since $\left \{V_\gamma, V_\vartheta \right \}$ are orthonormal at each point we easily find:
\[
{\rm trace}\langle d\Phi \cdot,\left ( {\rm trace}\langle
(fV^*),d\Phi \cdot \rangle d\Phi \cdot \right ) \rangle d\Phi \cdot= f\langle
V^*,V_\gamma \rangle V_\gamma+f\langle
V^*,V_\vartheta \rangle V_\vartheta\,.
\]
\textbf{Term X.}
\[
-2 |d\Phi|^2\, {\rm trace}\langle d\Phi \cdot,(fV^*) \rangle d\Phi \cdot=-4f\langle V_\gamma,V^*\rangle V_\gamma -4f\langle V_\vartheta,V^*\rangle V_\vartheta \,.
\]
\textbf{Term XI.}
\[
2 \langle d(fV^*),d\Phi\rangle \tau(\Phi)=-4\Big [2 f_\gamma \langle V^*,V_\gamma \rangle + f \langle \nabla^{\Phi}_{X_\gamma}V^*,V_\gamma \rangle 
+2 f_\vartheta \langle V^*,V_\vartheta \rangle+ f \langle \nabla^{\Phi}_{X_\vartheta}V^*,V_\vartheta \rangle \Big ]V_\eta\,. 
\]
Therefore the calculation of this term ends easily using \eqref{eq-nabla-V-tutti}.

\textbf{Term XII.}
\[
-|d\Phi|^2 \,\overline{\Delta}(fV^*)=-2 \overline{\Delta}(fV^*)\,.
\]
This computation was performed in Lemma\link\ref{lemma3}.

\textbf{Term XIII.} $|d\Phi|^4 (fV^*)=4fV^*\,.$

\vspace{2mm}

Now we are in the right position to complete the proof of the proposition. As for 
\eqref{eq-I2-fV-tutti}(i), we follow the lines of computation which we have just described and we obtain the $13$ terms I--XIII in the case that $V=f V_\gamma$. The result is 
\begin{eqnarray*}
 {\rm I}&=& \left[ (\lambda+3)^2 f -48 f_{\gamma \gamma}\right ]V_\gamma + \left [ 16f_{\gamma_\vartheta}\right ]V_\vartheta\\
 &&+ \left [ 4 \sqrt 2 (2\lambda+7)f_\gamma \right ]V_\nu +\left [ 4  (2\lambda+5)f_\gamma \right ]V_\eta\\
 {\rm II}&=&-\left [(\lambda+3)f\, V_\gamma+4 \sqrt 2 \,f_\gamma \,V_\nu+4 f_\gamma\,V_\eta \right ] \\
{\rm III}&=&  -8 f V_\gamma \\
 {\rm IV}&=&  4 f V_\gamma \\
 {\rm V}&=&  4 f V_\gamma \\
 {\rm VI}&=& - 4 f V_\gamma \\
 {\rm VII}&=& 0\\
 {\rm VIII}&=&(\lambda+3)  f V_\gamma  \\
 {\rm IX}&=&   f V_\gamma \\
 {\rm X}&=& - 4 f V_\gamma \\
 {\rm XI}&=&-8 f_\gamma V_\eta\\
 {\rm XII}&=&-2\left [(\lambda+3)f\, V_\gamma+4 \sqrt 2 \,f_\gamma \,V_\nu+4 f_\gamma\,V_\eta \right ] \\
 {\rm XIII}&=& 4 f V_\gamma \,.
 \end{eqnarray*} 
Adding up all these $13$ terms and simplifying we obtain \eqref{eq-I2-fV-tutti}(i). The proof of \eqref{eq-I2-fV-tutti}(ii)--(iv) is analogous and so we omit further details.
\end{proof}

\begin{proof}[Proof of Theorem\link\ref{Th-Main}] We recall that the Laplace operator $\Delta$ on $\t$ is given in \eqref{eq:definition-laplace} and we denote by $\lambda_i,\,i \in \n$, its spectrum. 

We know that the eigenvalues of $\Delta$ have the form $\lambda_i=4(m^2+n^2)$, where $m,n \geq 0$. Then it is convenient to define
\begin{eqnarray}\label{sottospazi-S-lambda}
 S^{\lambda_i}&=&\left \{ f_1V_\gamma : \Delta f_1= \lambda_i f_1 \right \} \oplus \left \{ f_2V_\vartheta : \Delta f_2= \lambda_i f_2 \right \}\nonumber\\
 &&\oplus \left \{ f_3V_\nu :\Delta f_3= \lambda_i f_3 \right \}
 \oplus \left \{ f_4V_\eta : \Delta f_4= \lambda_i f_4 \right \}
 \end{eqnarray}
As in \cite{LO},  $S^{\lambda_i} \perp S^{\lambda_j}$ if $i \neq j$ and $\oplus_{i=0}^{+\infty}\, S^{\lambda_i}$ is dense in
$\mathcal{C}\left( \Phi^{-1}T\s^4\right )$ (note that the scalar product which we use on sections of $ \Phi^{-1}T\s^4$ is the standard $L^2$-inner product). Moreover, using the explicit description \eqref{sottospazi-S-lambda}, it is easy to deduce from Proposition\link\ref{proposizione-I-esplicito-toro-sfera} that \textit{the operator $I_2$ preserves each of the $ S^{\lambda_i}$.} Indeed, this is a consequence of the fact that, if $f$ is an eigenfunction of $\Delta$ with eigenvalue $\lambda$, then the same is true for all its partial derivatives with respect to $\gamma,\, \vartheta$ because $\partial \slash \partial \gamma$ and $\partial \slash \partial \vartheta$ are Killing vector fields on the torus. 

By way of summary, we can compute the biharmonic index and nullity of $I_2$ restricted to each of the $ S^{\lambda_i}$'s and then add up the results to complete the proof of Theorem\link\ref{Th-Main}. 

First, let us examine the eigenvalue $\lambda_0=0$. We have  
\begin{equation*}\label{sottospazi-S-lambda-0}
 S^{\lambda_0}=\left \{ c_1\,V_\gamma \,\,:\,\, c_1 \in \R\right \}\oplus \left \{ c_2\,V_\vartheta \,\,:\,\, c_2\in \R\right \}\oplus \left \{ c_3\,V_\nu \,\,:\,\, c_3\in \R\right \} \oplus \left \{ c_4\,V_\eta \,\,:\,\, c_4\in \R\right \}
 \end{equation*}
and $\dim \left( S^{\lambda_0} \right)=4$. It follows by a direct application of Proposition\link\ref{proposizione-I-esplicito-toro-sfera} that the restriction of $I_2$ to $S^{\lambda_0}$ gives rise to the eigenvalues $\mu_1=- 16$ with multiplicity $1$ (eigenvector $V_\eta$), and $\mu_2=0$ with multiplicity equal to $3$ (eigenvectors $\left \{V_\gamma, V_\vartheta, V_\nu \right \}$). 

Then, we conclude that the contributions of this subspace to Index$_2(\Phi)$ and Null$_2(\Phi)$ are respectively $1$ and $3$.

Next, let us consider the case that $\lambda>0$. In this case, it is difficult to describe explicitly all the couples $(m,n)$ such $\lambda=4(m^2+n^2 )$ and so we proceed introducing a further, more suitable, decomposition.
 
More precisely, let us denote by $W_{\lambda}$ the corresponding eigenspace. In a similar fashion to \cite{BFO}, we decompose
\begin{equation}\label{sottospazi-W-lambda}
W_{\lambda}=W^{m,0}\oplus_{m,n \geq 1}W^{m,n}\oplus W^{0,n}\,,
 \end{equation}
where it is understood that in \eqref{sottospazi-W-lambda} we have to consider all the possible couples $(m,n)\in \n \times \n$ such that $\lambda=4(m^2+n^2)$. 
The subspaces of the type $W^{m,0}$ are $2$-dimensional and are spanned by the functions $\left \{\cos (m\gamma),\sin (m\gamma) \right \}$. Similarly, $W^{0,n}$ is $2$-dimensional and is generated by $\left \{\cos (n\vartheta),\sin (n\vartheta) \right \}$. Finally, the subspaces $W^{m,n}$, with $m,n \geq 1$, have dimension $4$ and are spanned by
\[
\left \{\cos (m\gamma)\cos(n\vartheta),\cos (m\gamma)\sin (n\vartheta),\sin (m\gamma)\cos(n\vartheta),\sin (m\gamma)\sin (n\vartheta)\right \} \,.
\]
Now it becomes natural to define
\begin{eqnarray*}\label{sottospazi-S-m,n}
 S^{m,n}&=&\left \{ f_1 V_\gamma : f_1 \in W^{m,n} \right \} \oplus \left \{ f_2 V_\vartheta : f_2 \in W^{m,n} \right \}\\
 &&\oplus \left \{ f_3 V_\nu : f_3 \in W^{m,n} \right \} \oplus \left \{ f_4 V_\eta : f_4 \in W^{m,n} \right \} \,.
 \end{eqnarray*}
All these subspaces are orthogonal to each other. Moreover, for any positive eigenvalue $\lambda_i$, we have
\begin{equation*}\label{decomposizioneS-lambda-in-coppie-m-n}
 S^{\lambda_i}= \oplus_{4(m^2+n^2)=\lambda_i}\,\,S^{m,n}\,.
 \end{equation*}
It follows easily from Proposition\link\ref{proposizione-I-esplicito-toro-sfera} that the operator $I_2$ preserves each of the subspaces $S^{m,n}$. Therefore, its spectrum can be computed by determing the eigenvalues of the matrices associated to the restriction of $I_2$ to each of the $S^{m,n}\,$'s. We shall see that $\dim \left ( S^{m,n}\right )$ is either $8$ or $16$ and so this approach enable us to provide a rather unified treatment, while the direct study of the $S^{\lambda_i}$'s would be much trickier. 

We point out that, in the rest of the proof, there are several long and tedious computations which can be conveniently carried out using a suitable software (we used Mathematica$^{\footnotesize \textregistered}$). 

We separate three cases:

\textbf{Case 1:} $S^{m,0}$, $m \geq1$.

In this case $\dim \left (S^{m,0} \right )=8$ and an orthonormal basis $\{e_i \}_{i=1 \ldots 8}$ of $S^{m,0}$ is given by:
\begin{eqnarray}\label{eq-ort-base-Sm0}
\Big \{ \frac{\sqrt{2}\cos (m\gamma)}{\pi}\,V_\gamma ,\frac{\sqrt{2}\sin (m\gamma)}{ \pi}\,V_\gamma,\frac{\sqrt{2}\cos (m\gamma)}{\pi} 
 V_\vartheta,\frac{\sqrt{2}\sin (m\gamma)}{\pi}\,V_\vartheta\\\nonumber
\frac{\sqrt{2}\cos (m\gamma)}{\pi}\,V_\nu ,\frac{\sqrt{2}\sin (m\gamma)}{\pi}\,V_\nu ,
 \frac{\sqrt{2}\cos (m\gamma)}{\pi}\,V_\eta,\frac{\sqrt{2}\sin (m\gamma)}{\pi}\,V_\eta \Big \} \,.
\end{eqnarray}
Using Proposition\link\ref{proposizione-I-esplicito-toro-sfera} with $\lambda=4 m^2$ and computing we find that the $(8\times 8)$-matrices which describe the operator $I_2$ with respect to the basis \eqref{eq-ort-base-Sm0}, i.e. $\left( I_2(e_i),e_j \right )$, are:
\begin{equation*}\label{matrici-S-m-0}
\left(
\begin{array}{cccc}
 16 m^2 \left(m^2+4\right) & 0 & 0 & 0 \\
 0 & 16 m^2 \left(m^2+4\right) & 0 & 0  \\
 0 & 0 & 16 \left(m^4+m^2\right) & 0  \\
 0 & 0 & 0 & 16 \left(m^4+m^2\right) \\
 0 & 8 \sqrt{2} m \left(4 m^2+2\right) & 0 & 0 \\
 -8 \sqrt{2} m \left(4 m^2+2\right) & 0 & 0 & 0  \\
 0 & 32 m^3 & 0 & 0  \\
 -32 m^3 & 0 & 0 & 0 
\end{array}
\right.
\end{equation*}
\begin{equation*}
\left.
\begin{array}{cccc}
  0 & -8 \sqrt{2} m
   \left(4 m^2+2\right) & 0 & -32 m^3 \\
  8 \sqrt{2} m
   \left(4 m^2+2\right) & 0 & 32 m^3 & 0 \\
  0 & 0 & 0 & 0 \\
  0 & 0 & 0 & 0 \\
  16 m^2
   \left(m^2+3\right) & 0 & 16 \sqrt 2 m^2 & 0 \\
 0 & 16 m^2
   \left(m^2+3\right) & 0 & 16 \sqrt 2 m^2 \\
  16 \sqrt 2 m^2 & 0 & 16 \left(m^4+m^2-1\right)
   & 0 \\
 0 & 16 \sqrt 2 m^2 & 0 & 16
   \left(m^4+m^2-1\right) \\
\end{array}
\right)
\end{equation*}
The characteristic polynomial is
\begin{equation}\label{pol-car8x8}
P(x)=\left (x-16(m^2+m^4)\right )^2 \,\big (P_3(x) \big )^2 \,,
\end{equation}
where
\begin{equation*}\label{eq-char-pol-6-Sm0}
P_3(x)=a_0+a_1 x+ a_2 x^2 + a_3 x^3 
\end{equation*}
and its coefficients are the following:
\begin{eqnarray*}\label{eq-coeff-pol-6-S-m-0}
a_0&=&-4096  m^2 (m-1)(m+1) (m^8- 3 m^6+ m^4+ 4 m^2-2 )\\\nonumber
a_1&=&256 m^2 (3 m^6+ 4 m^4+ 7 m^2  -9)\\\nonumber
a_2&=&-16 ( 3 m^4+ 8 m^2 -1 )\\\nonumber
a_3&=&1 \,.\nonumber
\end{eqnarray*}
Because all the roots of $P_3(x)$ are real, according to the Descartes rule we can conclude that the third order polynomial $P_3(x)$ possesses three positive roots provided that
\begin{equation}\label{eq-cartesio-regola-Sm-0}
a_0 <0 \,, \quad a_1>0 \,, \quad a_2 <0\,, \quad a_3 >0\,.
\end{equation}
Now, it is easy to check that \eqref{eq-cartesio-regola-Sm-0} is satisfied provided that $m \geq 2$.

By contrast, when $m=1$ we have $a_0=0$. In this case, it is easy to conclude that $P_3(x)$ has two positive roots and one root equal to $0$, with multiplicity $2$ as a root of the characteristic polynomial $P(x)$ in \eqref{pol-car8x8}. 

In summary, we have proved that the contributions of the subspaces $S^{m,0}$, $m \geq1$, to Null$_2(\Phi)$ and Index$_2(\Phi)$ are respectively $2$ and $0$.

\textbf{Case 2:} $S^{0,n}$, $n \geq1$.

Because of the symmetry of the map $\Phi$, the contribution of the subspaces $S^{0,n}$, $n \geq1$, to Null$_2(\Phi)$ and Index$_2(\Phi)$ is precisely as in Case $1$ above, i.e., $2$ for the nullity and $0$ for the index.

\textbf{Case 3:} $S^{m,n}$, $m,n \geq1$.

This is the case which requires the biggest computational effort. In this case $\dim \left (S^{m,n} \right )=16$ and an orthonormal basis $\{e_i \}_{i=1 \ldots 16}$ of $S^{m,n}$ is given by:
\begin{eqnarray*}
\Big \{ \frac{2}{\pi}\cos (m\gamma)\cos (n\vartheta)V_\gamma ,\frac{2}{\pi}\cos (m\gamma)\sin (n\vartheta)V_\gamma,\frac{2}{\pi}\sin (m\gamma)\cos (n\vartheta)V_\gamma,\frac{2}{\pi}\sin (m\gamma)\sin (n\vartheta)V_\gamma,\\
 \frac{2}{\pi}\cos (m\gamma)\cos (n\vartheta)V_\vartheta ,\frac{2}{\pi}\cos (m\gamma)\sin (n\vartheta)V_\vartheta,\frac{2}{\pi}\sin (m\gamma)\cos (n\vartheta)V_\vartheta,\frac{2}{\pi}\sin (m\gamma)\sin (n\vartheta)V_\vartheta,\\
 \frac{2}{\pi}\cos (m\gamma)\cos (n\vartheta)V_\nu ,\frac{2}{\pi}\cos (m\gamma)\sin (n\vartheta)V_\nu,\frac{2}{\pi}\sin (m\gamma)\cos (n\vartheta)V_\nu,\frac{2}{\pi}\sin (m\gamma)\sin (n\vartheta)V_\nu,\\
\frac{2}{\pi}\cos (m\gamma)\cos (n\vartheta)V_\eta ,\frac{2}{\pi}\cos (m\gamma)\sin (n\vartheta)V_\eta,\frac{2}{\pi}\sin (m\gamma)\cos (n\vartheta)V_\eta,\frac{2}{\pi}\sin (m\gamma)\sin (n\vartheta)V_\eta \Big \}\,.
\end{eqnarray*}
As an application of Proposition\link\ref{proposizione-I-esplicito-toro-sfera} with $\lambda=4(m^2+n^2)$, with the aid of Mathematica$^{\footnotesize \textregistered}$ we compute the $(16\times 16)$-matrices $\left ( I_2(e_i),e_j \right )$ and find that their characteristic polynomial is
\[
P(x)= \big [c_0+c_1 x+c_2 x^2 +c_3 x^3 + c_4 x^4\big ]^4 = \big [ Q_4(x) \big ]^4\,,
\]
where the coefficients of the fourth order polynomial $Q_4(x)$ are:
\begin{eqnarray*}
c_0&=&65536 \Big [m^{16} + n^{16} + 8 ( m^{14} n^2 + m^2 n^{14}) + 28 (m^{12} n^4 + m^4 n^{12}) + 
 56 (m^{10} n^6 + m^6 n^{10}) \\&&+ 70 m^8 n^8 - 3 (m^{14} + n^{14}) - 
 21 ( m^{12} n^2 + m^2 n^{12}) - 63 ( m^{10} n^4 + m^4 n^{10}) \\&&- 
 105 ( m^8 n^6 + m^6 n^8) + 16 ( m^{10} n^2 + m^2 n^{10}) + 
 64 (m^8 n^4 + m^4 n^8) + 96 m^6 n^6 \\&&+ 7 (m^{10} + n^{10}) - 
 21 ( m^8 n^2 + m^2 n^8) - 98 (m^6 n^4 + m^4 n^6) - 3 ( m^8 + n^8) \\&&+ 
 12 (m^6 n^2 + m^2 n^6) + 94 m^4 n^4 - 4 ( m^6 + n^6) - 
 12 ( m^4 n^2 + m^2 n^4) \\&&+ 2 (m^4 + n^4) + 12 m^2 n^2\Big ] 
 \end{eqnarray*}
 \begin{eqnarray*}
c_1&=& -4096 \Big[4 (m^{12} + n^{12}) + 24 (m^{10} n^2 + m^2 n^{10}) + 
 60 ( m^8 n^4 + m^4 n^8) + 80 m^6 n^6  \\&&+ 3 ( m^{10} + n^{10}) + 
 15 ( m^8 n^2 + m^2 n^8) + 30 ( m^6 n^4 + m^4 n^6) + 
 15 ( m^8 + n^8)  \\&&- 36 ( m^6 n^2 + m^2 n^6) - 
 102 m^4 n^4 + m^6 + n^6 + 11 ( m^4 n^2 + m^2 n^4) - 
 15 ( m^4 + n^4)  \\&&- 46 m^2 n^2 + 2 (m^2 + n^2)\Big ]
  \end{eqnarray*}
 \begin{eqnarray*}
c_2&=&256 \Big [6 ( m^8 + n^8) + 24 ( m^6 n^2 + m^2 n^6) + 36 m^4 n^4 + 
 15 (m^6 + n^6) \\&&+ 45 ( m^4 n^2 + m^2 n^4)  + 14 ( m^4 + n^4) + 
 44 m^2 n^2 - 10 ( m^2 + n^2)\Big ] \\
c_3&=&-16 \big[4 (m^4 + n^4) + 8 m^2 n^2 + 9 (m^2 + n^2) - 1\big ] \\
c_4&=&1 \,. 
\end{eqnarray*}
Next, with the methods used in \cite{MOR3}, it is not difficult to show that, if $m \geq 2,n \geq 1$ or $m \geq 1,n \geq 2$, then
\[
c_0 >0\, , \quad  \quad c_1 < 0\,,\quad c_2 > 0 \,,\quad  c_3 <0 \,, \quad c_4>0\,.
\]
Because all the roots of $Q_4(x)$ are real, we conclude that the fourth order polynomial $Q_4(x)$ admits $4$ positive roots in these cases.

By contrast, when $m=n=1$ we have:
\[
c_0 =0\, , \quad  \quad c_1 < 0\,,\quad c_2 > 0 \,,\quad  c_3 <0 \,, \quad c_4>0\,.
\]
It follows easily that, in this case, $Q_4(x)$ has $3$ positive roots and one root equal to $0$. This last root corresponds to the zero eigenvalue for $I_2$, with multiplicity equal to $4$.

In summary, we have proved that the contributions of the subspaces $S^{m,n}$, $m ,n\geq1$, to Null$_2(\Phi)$ and Index$_2(\Phi)$ are respectively $4$ and $0$ and this ends Case 3.

Adding up the results of $S^{\lambda_0}$, $\lambda_0=0$, with those of Cases 1,2,3 we conclude that
\[
{\rm Null}_2(\Phi)=3+2+2+4=11\,, \quad {\rm Index}_2(\Phi)=1+0+0+0=1
\]
and so the proof of Theorem\link\ref{Th-Main} is completed.
\end{proof}

\section{The study of Ker$(I_2)$}\label{section-null}
Let $\phi : (M,g) \to (N,h)$ be a biharmonic map. First, we recall a basic fact about Ker$(I_2)$, that is: {\em  if $\{\phi_t\}$ is a variation of $\phi$ by means of biharmonic maps, then $V= \left . \frac{d}{dt}\right |_{t=0} \phi_{t}$ belongs to Ker$(I_2)$}. In fact,  for an arbitrary $W\in {\mathcal C}(\phi^{-1}TN)$, we can consider a variation $\{W_t\}$ such that $W_0=W$ and $W_t\in {\mathcal C}(\phi^{-1}_tTN)$. Then we define the two parameter variation of $\phi$  by
\[
\phi_{t,s}(x) = {\rm exp}_{\phi_t(x)}(sW_t(x))\,.
\]
With respect to $\{\phi_{t,s}\}$ we have
\[
V= \left . \frac{d}{dt}\right |_{t=0} \phi_{t,0}\,,  \qquad
W= \left . \frac{d}{ds}\right |_{s=0} \phi_{0,s} \,.
\]
As $\phi_t$ is biharmonic for all $t$, 
\[
\left . \frac{\partial}{\partial  s}\right |_{s=0}  E_2 (\phi_{t,s})=0
\]
and consequently
\[
\left . \frac{\partial^2}{\partial t \partial s}\right |_{(t,s)=(0,0)}  E_2 (\phi_{t,s})=(I_2(V),W)=0\,.
\]
We conclude that $V$ belongs to Ker$(I_2)$.
In particular,  if $\{\phi_t\}$ is given by composing $\phi$ with a one parameter family of isometries of the domain or the codomain, the above properties imply 
\[
\{d\phi(X)\colon X\in {\mathcal C}(TM), X \text{ is Killing}\}\subset {\rm Ker}(I_2)
\]
and
\[
\{Z\circ \phi \colon Z\in {\mathcal C}(TN), Z \text{ is Killing}\}\subset {\rm Ker}(I_2)\,.
\]

In this section we show that Ker$(I_2)$ is the orthogonal sum of a $10$-dimensional subspace spanned by Killing vector fields as above and a $1$-dimensional subspace spanned by $V_\nu$. 

It is easy to describe the space of Killing vector fields on $\t$: it has dimension $2$ and it is spanned by $\{X_\gamma, X_\vartheta \}$. 

As for the target, we know that the space of Killing vector fields on $\s^n$ has dimension $n(n+1)\slash 2$. In particular, a base for this subspace of $\mathcal{C}\left(T\s^4\right)$ can be obtained by restriction of the following $10$ vector fields on $\R^5$: 
\begin{equation*}\label{Killing-fields-on-R5}
\begin{array}{ll}
Z_1= \left (-y^2,y^1,0,0,0 \right ) \,;\quad &Z_2= \left (0,0,-y^4,y^3,0 \right ) \,; \\
Z_3= \left (-y^4,0,0,y^1,0 \right ) \,;\quad &Z_4= \left (0,-y^3,y^2,0,0 \right ) \,; \\
Z_5= \left (-y^3,0,y^1,0,0 \right ) \,;\quad &Z_6= \left (0,-y^4,0,y^2,0 \right ) \,;\\
Z_7= \left (0,-y^5,0,0,y^2 \right ) \,;\quad &Z_8= \left (-y^5,0,0,0,y^1 \right ) \,; \\
Z_9= \left (0,0,0,-y^5,y^4 \right ) \,;\quad &Z_{10}= \left (0,0,-y^5,0,y^3 \right ) .
\end{array}
\end{equation*}
Then, we define vector fields $V_i\in \mathcal{C}\left(\Phi^{-1}T\s^4\right)$ as follows:
\begin{equation}\label{Killing-along-Phi}
V_i(P)=Z_i(\Phi(P))\,, \qquad i=1, \dots 10 \,.
\end{equation}
Now, we are in the right position to state the main result of this section.
\begin{theorem}\label{Th-Killing}
\begin{equation}\label{eq-spezz-nullity}
{\rm Ker}(I_2)= W^{(10)} \oplus W^{(1)}\,,
\end{equation}
where $W^{(10)}$ is a $10$-dimensional subspace spanned by the vector fields $V_i\in \mathcal{C}\left(\Phi^{-1}T\s^4\right)$ defined in \eqref{Killing-along-Phi}, while $\dim W^{(1)}=1$ and $W^{(1)}$ is spanned by $V_\nu$. 
\end{theorem}
\begin{proof} A computation, using the notations in \eqref{eq-def-V-vari}, shows that:
\begin{eqnarray}\label{formule-esplicite-Vi-killing}
V_1&=&\frac{1}{2} V_\gamma\nonumber \\\nonumber
V_2&=&\frac{1}{2} V_\vartheta \\
V_3&=&\left [\frac{1}{2}\sin \gamma \sin \vartheta\right ] V_\gamma+\left [\frac{1}{2}\cos \gamma \cos \vartheta\right ] V_\vartheta -\left [\frac{1}{\sqrt 2} \cos \gamma \sin \vartheta \right ]V_\nu \\\nonumber
V_4&=&-\left [\frac{1}{2}\cos \gamma \cos \vartheta\right ] V_\gamma-\left [\frac{1}{2}\sin \gamma \sin \vartheta\right ] V_\vartheta-\left [\frac{1}{\sqrt 2}\sin \gamma \cos \vartheta\right ] V_\nu \\\nonumber
V_5&=&\left [\frac{1}{2}\sin \gamma \cos \vartheta\right ] V_\gamma-\left [\frac{1}{2}\cos \gamma \sin \vartheta \right ]V_\vartheta -\left [\frac{1}{\sqrt 2}\cos \gamma \cos \vartheta\right ] V_\nu \\\nonumber
\end{eqnarray}
\begin{eqnarray}\label{formule-esplicite-Vi-killing-bis}
V_6&=&-\left [\frac{1}{2}\cos \gamma \sin \vartheta\right ] V_\gamma+\left [\frac{1}{2}\sin \gamma \cos \vartheta \right ]V_\vartheta -\left [\frac{1}{\sqrt 2}\sin \gamma \sin \vartheta\right ] V_\nu \nonumber\\\nonumber
V_7&=&-\left [\frac{1}{\sqrt 2}\cos \gamma \right ] V_\gamma-\left [\frac{1}{2}\sin \gamma  \right ]V_\nu -\left [\frac{1}{\sqrt 2}\sin \gamma \right ] V_\eta \\
V_8&=&\left [\frac{1}{\sqrt 2}\sin \gamma \right ] V_\gamma-\left [\frac{1}{2}\cos \gamma \right ]V_\nu -\left [\frac{1}{\sqrt 2}\cos \gamma \right ] V_\eta \\\nonumber
V_9&=&-\left [\frac{1}{\sqrt 2}\cos \vartheta \right ] V_\vartheta+\left [\frac{1}{2}\sin \vartheta  \right ]V_\nu -\left [\frac{1}{\sqrt 2}\sin \vartheta \right ] V_\eta \\\nonumber
V_{10}&=&\left [\frac{1}{\sqrt 2}\sin \vartheta \right ] V_\vartheta+\left [\frac{1}{2}\cos \vartheta \right ]V_\nu -\left [\frac{1}{\sqrt 2}\cos \vartheta \right ] V_\eta \,. \nonumber
\end{eqnarray}
From this it is easy to check that, for all $1\leq i,j \leq 10,\, i \neq j$, we have
\begin{equation}\label{ortogonality}
\big( V_i, V_\nu \big )=0  \quad {\rm and}\quad \big( V_i, V_j \big )=0
\end{equation}
and so the $V_j$'s are linearly independent. Therefore, they span a $10$-dimensional subspace $W^{(10)}$ of ${\rm Ker}(I_2)$. Now, the statement \eqref{eq-spezz-nullity} in Theorem\link\ref{Th-Killing} is an immediate consequence of the fact that $ {\rm Null}_2(\Phi)=11$ and $V_\nu \in {\rm Ker}(I_2)$, as shown in the proof of Theorem\link\ref{Th-Main}. 
\end{proof}
\begin{remark}
(i) The contribution to the nullity of $\Phi$ given by the two Killing fields $\{X_\gamma, X_\vartheta \}$ on the domain $\t$ is included in $W^{(10)}$. Indeed, the first two equalities in \eqref{formule-esplicite-Vi-killing} show that both $d\varphi(X_\gamma)=V_\gamma=2 V_1$ and $d\varphi(X_\vartheta)=V_\vartheta=2 V_2$ belong to $W^{(10)}$.

(ii) In the notation of the proof of Theorem\link\ref{Th-Main}, the vector fields $V_3,V_4,V_5,V_6$ belong to $S^{1,1}$. Using Proposition\link\ref{proposizione-I-esplicito-toro-sfera} with $\lambda=8$, together with the explicit equalities \eqref{formule-esplicite-Vi-killing}--\eqref{formule-esplicite-Vi-killing-bis}, it is possible to check directly that $I_2(V_i)=0, \,i=3,\ldots,6$. Similarly, using $\lambda=4$, the same is true for $V_7,V_8$, which belong to $S^{1,0}$, and $V_9,V_{10} \in S^{0,1}$. 
\end{remark}

Next, we shall compute the higher order derivatives for a natural variation $\{\Phi_t\}$ such that $\partial \Phi_t \slash \partial t \big |_{t=0}=V_\nu$, and from this we shall deduce that \eqref{high-order-inst} occurs with $k=4$.
\begin{theorem}
Let $\Phi_t:\t \to \s^4\subset \R^5$ be defined by
\begin{equation}\label{def-Phit-Vnu}
\Phi_t= \frac{1}{\sqrt{1+t^2}}\left (\Phi + t V_\nu\right )\,.
\end{equation}
Then $\partial \Phi_t \slash \partial t \big |_{t=0}=V_\nu$ and we have:
\begin{eqnarray}\label{derivate-E2-Phit-Vnu}
\left. \frac{d^\ell}{dt^\ell}\right |_{t=0} E_2 (\Phi_t) &=&0\,, \quad \ell=1,2,3\,. \\\nonumber
\left. \frac{d^4}{dt^4}\right |_{t=0} E_2 (\Phi_t) &=&-48 \pi^2 <0 \,.
\end{eqnarray}
\end{theorem}
\begin{proof}
Let us consider the variation $\Phi_t:\t \to \s^4 \subset \R^5$ defined in \eqref{def-Phit-Vnu}. It is immediate to check that $\partial \Phi_t \slash \partial t \big |_{t=0}=V_\nu$. Then we have to compute explicitly the tension field of $\Phi_t$. To this purpose, we observe that
\begin{equation*}\label{tension-Phit-Vnu}
 \tau (\Phi_t)=- \Delta \Phi_t + \big | d \Phi_t \big |^2 \Phi_t \,,
 \end{equation*} 
where $\Delta$ was given in \eqref{eq:definition-laplace}. Now, a routine computation yields:
\begin{eqnarray*}
\tau (\Phi_t)&=&\Big (-\frac{\left(\sqrt{2} \,t+1\right) \cos \gamma }{\left(t^2+1\right)^{3/2}},-\frac{\left(\sqrt{2}\, t+1\right) \sin \gamma }{\left(t^2+1\right)^{3/2}},\frac{\left(\sqrt{2}\, t-1\right) \cos \vartheta }{\left(t^2+1\right)^{3/2}},\\
&&\frac{\left(\sqrt{2} \,t-1\right) \sin \vartheta }{\left(t^2+1\right)^{3/2}},\frac{\sqrt{2} \left(2 t^2+1\right)}{\left(t^2+1\right)^{3/2}}\Big )\,.
\end{eqnarray*}
From this we obtain
\[
\big | \tau (\Phi_t)\big|^2= \frac{4+8t^2}{(t^2+1)^2} \,.
\]
Finally,
\[
E_2(\Phi_t)= \frac{1}{2} \int_{\t} \big | \tau (\Phi_t)\big|^2 \,dV_{\t}=\frac{2+4t^2}{(t^2+1)^2} \,\pi^2
\]
from which \eqref{derivate-E2-Phit-Vnu} follows readily.
\end{proof}
\section{Further studies: composition and $p$-harmonic index and nullity of $\varphi$}\label{Composition}

So far, we computed the biharmonic index and the biharmonic nullity of $\Phi=i \circ \varphi$. 
The aim of the first part of this section is to carry out a detailed study of a natural sub-class of variations, i.e., variations of the form $\Phi_{s,t}=i \circ \varphi_{s,t}$. As a completion of this analysis, in the second part of the section we shall compute the $p$-harmonic index and nullity of $\varphi$. 

Now, more generally, let $\tilde{\Phi}=\tilde{\varphi} \circ i$, where $\tilde{\varphi}: M^m \to \s^{n-1}(\tfrac{1}{\sqrt 2})$, $m \geq 1,\,n \geq 3$, is a minimal immersion and $i:\s^{n-1}(\frac{1}{\sqrt 2}) \to \s^n$, is the biharmonic small hypersphere. The map $\tilde{\Phi}$ is proper biharmonic and we shall determine a general \textit{sufficient} condition which ensures that the second variation of $\tilde{\Phi}$ is nonnegatively defined on $\mathcal{C}\big (\tilde{\varphi}^{-1}T\s^{n-1} ( \tfrac{1}{\sqrt 2} )\big )$.

To this purpose, we study $2$-parameter variations
\[
\tilde{\Phi}_{s,t}=i \circ \tilde{\varphi}_{s,t} \,,
\]
where 
$$\tilde{\varphi}_{s,t}:M^m \to \s^{n-1} ( \tfrac{1}{\sqrt 2})$$
is a $2$-parameter variation of $\tilde{\varphi} $.

Let $V,W$ be vector fields associated to $\tilde{\Phi}_{s,t}$ as in \eqref{V-W}. We denote the subspace of these vector fields as 
\begin{equation*}
\mathcal{W}^{\tilde{\varphi}} \subset \mathcal{C}\big (\tilde{\Phi}^{-1}(T \s^n)\big ) 
\end{equation*}
(note that, geometrically, there is a natural identification between $\mathcal{W}^{\tilde{\varphi}}$ and $\mathcal{C}\big (\tilde{\varphi}^{-1}T\s^{n-1}(\frac{1}{\sqrt 2}) \big )$).
Now,
\[
E_2\left (\tilde{\Phi}_{s,t} \right )= \frac{1}{2} \int_{M} \left |\tau \left (\tilde{\Phi}_{s,t} \right ) \right |^2 dV_M
\]
and we observe that, in this case, 
\begin{equation}\label{composition}
\left |\tau \left (\tilde{\Phi}_{s,t} \right ) \right |^2=\left |d i\big (\tau (\tilde{\varphi}_{s,t})\big )+ {\rm trace}\nabla d i (d \tilde{\varphi}_{s,t} \cdot,d \tilde{\varphi}_{s,t} \cdot) \right |^2=|\tau \big(\tilde{\varphi}_{s,t}\big)|^2+ |d \tilde{\varphi}_{s,t}|^4
\end{equation}
where the last equality is a consequence of the fact that $i$ is an isometric immersion and the two terms are orthogonal. 

As a consequence of \eqref{composition}, for all $ V,W \in \mathcal{W}^{\tilde{\varphi}}$ we have:
\[
 (I_2^{\tilde{\Phi}} (V),W)= (I_2^{\tilde{\varphi}} (V),W) + 2 (J^{\tilde{\varphi}}_{(4)}(V),W)=((J^{\tilde{\varphi}})^2 (V),W) + 2 \,(J^{\tilde{\varphi}}_{(4)}(V),W)\,,
 \]
where $J^{\tilde{\varphi}}_{(p)}$ is the Jacobi operator associated to the $p$-energy and we write $J^{\tilde{\varphi}}$ for $J^{\tilde{\varphi}}_{(2)}$. 

Now, let $\pi:\mathcal{C}\big (\tilde{\Phi}^{-1}(T \s^n)\big )\to \mathcal{W}^{\tilde{\varphi}}$ be the orthogonal projection and denote 
\begin{equation}\label{I2Phitilde}
I_2^{\tilde{\Phi},\pi}(V)=\pi \big (I_2^{\tilde{\Phi}} (V)\big )=(J^{\tilde{\varphi}})^2 (V) + 2 \,J^{\tilde{\varphi}}_{(4)}(V)\quad \forall V\in \mathcal{W}^{\tilde{\varphi}}\,.
\end{equation}

We want to study the quadratic form  
\begin{equation}\label{Hessian-varphi}
H(E_2)_{\tilde{\Phi}} (V,W)= (I_2^{\tilde{\Phi},\pi} (V),W)\quad \forall\,V,W \in \mathcal{W}^{\tilde{\varphi}}\,.
\end{equation}
Our general result in this context is:
\begin{theorem}\label{Th-composition}
Suppose that no eigenvalue of $J^{\tilde{\varphi}}$ belongs to the open interval $(-2m,0)$. Then the quadratic form \eqref{Hessian-varphi} is nonnegatively defined, i.e.,
\begin{equation}\label{I2varphi-nonnegative}
(I_2^{\tilde{\Phi},\pi} (V),V) \geq 0 \quad \forall\,V \in \mathcal{W}^{\tilde{\varphi}}\,.
\end{equation}
\end{theorem}
\begin{proof} We know that
\begin{eqnarray}\label{Jacobi-operator}
 J^{\tilde{\varphi}}(V)&=&\overline{\Delta}V + {\rm trace}\Big ({\rm Riem}^{\s^{n-1}(1\slash \sqrt 2)}\big(d\tilde{\varphi}(\cdot),V \big )d\tilde{\varphi}(\cdot) \Big )\nonumber\\
 &=&\overline{\Delta}V + {\rm trace}\Big ( 2 \big [-\langle d\tilde{\varphi}(\cdot) , d\tilde{\varphi}(\cdot)\rangle V+\langle V , d\tilde{\varphi}(\cdot)\rangle d\tilde{\varphi}(\cdot) \big ] \Big )\\
 &=&\overline{\Delta}V + {\rm trace}\Big ( 2 \big [-m V+ V^{\top} \big ] \Big )\nonumber\,,
\end{eqnarray} 
where, for $V\in \mathcal{W}^{\tilde{\varphi}}$, $V^{\top}$ and $V^\perp$ are respectively the tangent and the normal component.

The expression of $J^{\tilde{\varphi}}_{(p)}$ was computed in \cite{Nagano} for a generic harmonic map $\tilde{\varphi}$. Because, in our context, $\tilde{\varphi}$ is a minimal immersion, (1.9) of \cite{Nagano} simplifies and gives:
\begin{equation}\label{Jp-relation-J}
J^{\tilde{\varphi}}_{(p)}(V)=(p-2)m^{\frac{p-4}{2}}\, d^* \big (\langle d V, d \tilde{\varphi} \rangle d \tilde{\varphi} \big )+m^{\frac{p-2}{2}}\, J^{\tilde{\varphi}}(V) \,.
\end{equation}
Now, using \eqref{Jp-relation-J} with $p=4$ in \eqref{I2Phitilde} we deduce that
\begin{equation}\label{I2Phitilde-bis}
I_2^{\tilde{\Phi},\pi}(V)=(J^{\tilde{\varphi}})^2 (V) + 2m \,J^{\tilde{\varphi}}(V)+4\,d^* \big (\langle d V, d \tilde{\varphi} \rangle d \tilde{\varphi} \big ) \quad \forall V\in \mathcal{W}^{\tilde{\varphi}}\,.
\end{equation}  
Since $\tilde{\varphi}$ is an immersion, there exists a unique vector field $\xi \in \mathcal{C}(TM)$ such that 
\[
d\tilde{\varphi}(\xi)=V^{\top} \,.
\]
We prove that \begin{equation}\label{d*-calcolo-xi}
\Big (d^* \big (\langle d V, d \tilde{\varphi} \rangle d \tilde{\varphi} \big ),V \Big)= \int_{M}({\rm div} \xi)^2 dV_M \,.
\end{equation}
Indeed, let $P \in M$ be arbitrarily fixed. We use a geodesic frame field $\{X_i\}$ around $P$. Then, at $P$ we have:
\begin{eqnarray*}
\langle d V, d \tilde{\varphi} \rangle&=& \sum_{i=1}^m \langle \nabla_{X_i}^{\tilde{\varphi}} (V^{\top}+V^\perp), d \tilde{\varphi}(X_i) \rangle \\
&=&\sum_{i=1}^m \Big \{ \langle\nabla^M_{X_i} \xi,X_i \rangle + \langle B(X_i,\xi),X_i \rangle\\
&&+\langle\nabla^\perp_{X_i}V^\perp,X_i \rangle - \langle A_{V^\perp}(X_i),X_i \rangle \Big \}\\
&=&\sum_{i=1}^m \Big \{ \langle\nabla^M_{X_i} \xi,X_i \rangle- \langle A_{V^\perp}(X_i),X_i \rangle \Big \}\\
&=& {\rm div}(\xi) -\sum_{i=1}^m  \langle B(X_i,X_i),V^\perp \rangle \\
&=& {\rm div}(\xi)-m \langle H, V^\perp \rangle = {\rm div}(\xi)\,,
\end{eqnarray*}
where $B$ and $A$ are the second fundamental form and the shape operator of $M$.
Thus
\[
\Big (d^* \big (\langle d V, d \tilde{\varphi} \rangle d \tilde{\varphi} \big ),V \Big)=\int_{M}\left (\langle d V, d \tilde{\varphi} \rangle \right )^2 dV_M  = \int_{M}({\rm div} \xi)^2 dV_M 
\]
and so \eqref{d*-calcolo-xi} is verified.

Now, we can end the proof. Let $V\in \mathcal{W}^{\tilde{\varphi}}$ and assume $J^{\tilde{\varphi} }(V)= \mu V$. Then it follows from \eqref{I2Phitilde-bis} and \eqref{d*-calcolo-xi} that
\[
\big( I_2^{\tilde{\Phi},\pi}(V),V \big )=\int_{M} \big \{(\mu^2 + 2m \mu)|V|^2+ 4({\rm div} \xi)^2 \big \} dV_M \,,
\]
and from this the conclusion of Theorem\link\ref{Th-composition} follows immediately.
\end{proof}

\begin{example}
Theorem~\ref{Th-composition} can be applied when $\tilde{\varphi}$ is the totally geodesic embedding of $\s^m(1/\sqrt{2})$ in $\s^{n-1}(1/\sqrt{2})$, $m=1,2$, $n\geq 3$. In fact, in these cases, the first eigenvalues of $J^{\tilde{\varphi}}$ are $-2m$ and $0$ which do not belong to the open interval $(-2m,0)$ (see \cite{ER}, Appendix $1$).
\end{example}

Theorem\link\ref{Th-composition} suggests to investigate the spectrum of $J^{\tilde{\varphi} }$, or, more in detail, to compute the index and the nullity of $J^{\tilde{\varphi}}$.

The stability of minimal and CMC immersions with respect to the \textit{volume functional} has been widely studied in the literature (see \cite{Alias,Barbosa} and \cite{Ciraolo} for more recent developments). On the other hand, any minimal immersion is also a \textit{$p$-harmonic map} and so it is natural to study its second variation as a critical point of the \textit{$p$-energy functional} \eqref{p-energy-def}.

There are not many papers where index and nullity computations in this context have been carried out. Some interesting results of this type were obtained in \cite{Wei} where, for instance, the $p$-harmonic index of the identity map was computed when $M$ is an Einstein manifold. Nevertheless, not much has been done when the map is not the identity and so our first goal is to compute the \textit{p-harmonic index} and the \textit{p-harmonic nullity} of the minimal Clifford torus
\begin{eqnarray*}\label{eq-mapdatoroasfera-bis} \nonumber
\varphi:\t\quad &\to &\s^3(\tfrac{1}{\sqrt 2})\subset \R^4\\
    \left ( \gamma, \vartheta \right )  &\mapsto &\left (\frac{1}{2} \cos \gamma,  \frac{1}{2} \sin \gamma,\frac{1}{2} \cos \vartheta,  \frac{1}{2} \sin \vartheta  \right ) \,,\quad 0\leq \gamma, \vartheta\leq 2 \pi \,.
\end{eqnarray*}
We recall that ${\rm Index}_{\rm p-harm}(\varphi)$ and ${\rm Null}_{\rm p-harm}(\varphi)$ are defined precisely as in \eqref{Index-definition} and \eqref{Nullity-definition} respectively, with $I_2$ replaced by $J^\varphi_{(p)}$.
As in the case of $I_2$, when the context is clear we shall simply write $J$, $J_{(p)}$ instead of $J^\varphi$, $J^\varphi_{(p)}$ respectively. Also, when $p=2$, we write $_{\rm harm}$ instead of $_{\rm 2-harm}$.

Our first result in this setting is:
\begin{theorem}\label{Th-varphi}
Let $\varphi:\t \to \s^3(\tfrac{1}{\sqrt 2}) $ be the minimal Clifford torus. Then
\begin{equation*}\label{eq-index+nullity-main}
\begin{array}{llll}
{\rm (i)}&{\rm Index}_{\rm harm}(\varphi)&=& 4 \,; \\ 
{\rm (ii)}&{\rm Null}_{\rm harm}(\varphi)&=& 7 \,. 
\end{array}
\end{equation*}
\end{theorem}
\begin{proof}
The proof follows the lines of the proof of Theorem\link\ref{Th-Main}, so we just point out the relevant intermediate steps. First, a computation as in the proof of Proposition\link\ref{proposizione-I-esplicito-toro-sfera} gives:
\begin{proposition}\label{proposizione-J-esplicito-toro-sfera} Let $f\in C^{\infty}\left ( \t \right )$ be an eigenfunction of $\Delta$ with eigenvalue $\lambda$. Then
\begin{equation}\label{eq-J-fV-tutti}
\begin{array}{llll}
{\rm (i)}&J \left ( f\, V_\gamma \right )&=&\lambda f \,V_\gamma + 4 \sqrt 2 f_{\gamma }\, V_\nu \\
&&&\\
{\rm (ii)}&J \left ( f\, V_\vartheta \right )&=& \lambda f\, V_\vartheta - 4 \sqrt 2 f_{\vartheta }\, V_\nu  \\
&&&\\
{\rm (iii)}&J \left ( f\, V_\nu \right )&=&-4 \sqrt 2 f_{\gamma}\, V_{\gamma}+4 \sqrt 2 f_{\vartheta}\, V_{\vartheta}+\lambda \,f \,V_\nu\,.
\end{array}
\end{equation}
\end{proposition}
The subspaces $S^0$ and $S^{m,n}$ are defined precisely as in the proof of Theorem\link\ref{Th-Main}, just omitting $V_\eta$ and working in $\R^4$. So, now, the subspace $S^0$ is spanned $\{V_\gamma,V_\vartheta,V_\nu \}$ and $\dim (S^0)=3$. It follows immediately from Proposition\link\ref{proposizione-J-esplicito-toro-sfera} with $\lambda=0$ that $S^0$ belongs to ${\rm Ker}(J)$ and thus its contribution to ${\rm Null}_{\rm harm}(\varphi)$ is equal to $3$.

As for $S^{m,0}$, $m \geq1$, we construct the $6\times 6$-matrices $(J(e_i),e_j)$ and find that their associated characteristic polynomial is
\[
P(x)=(-4 m^2 + x)^2 (16 m^4 + x^2 - 8 m^2 (4 + x))^2\,.
\]
Then we see that, when $m \geq2$, all the eigenvalues are positive. By contrast, when $m=1$, we have one negative eigenvalue $\mu_1=4-4\sqrt 2$ with multiplicity $2$. The contribution of the subspaces $S^{0,n}$ is analogous. In summary, the subspaces $S^{m,0}$ and $S^{0,n}$, $m,n \geq1$ do not contribute to ${\rm Null}_{\rm harm}(\varphi)$ and give a contribution equal to $4$ to ${\rm Index}_{\rm harm}(\varphi)$. 
Next, one studies the subspaces $S^{m,n}$, $m,n\geq1$, and finds that there is no contribution to ${\rm Index}_{\rm harm}(\varphi)$. By contrast, the eigenvalue $\mu=0$ appears in the study of $S^{1,1}$ with multiplicity $4$. Adding up all the contributions we end readily the proof of Theorem\link\ref{Th-varphi}.
\end{proof}
\begin{remark}In order to complete the analysis which we carried out during the proof of Theorem\link\ref{Th-varphi}, we point out that a basis for the $4$-dimensional eigenspace $\mathcal{W}_{\mu_1}$ associated to $\mu_1$ is $\{W_1,W_2,W_3,W_4 \}$, where
\begin{eqnarray}\label{base-autovett-varphi}
W_1&=&\frac{\sqrt 2}{\pi} \Big [\cos \gamma\, V_\gamma + \sin \gamma \,V_\nu \Big ] \\\nonumber
W_2&=&\frac{\sqrt 2}{\pi} \Big [-\sin \gamma\, V_\gamma + \cos \gamma \,V_\nu \Big ] \\\nonumber
W_3&=&\frac{\sqrt 2}{\pi} \Big [-\cos \vartheta\, V_\vartheta + \sin \vartheta \,V_\nu \Big ] \\\nonumber
W_4&=&\frac{\sqrt 2}{\pi} \Big [\sin \vartheta\, V_\vartheta + \cos \vartheta \,V_\nu \Big ]\,. \nonumber
\end{eqnarray}
Let $\mathcal{W}_{{\rm conf}}\subset \mathcal{C}\Big(\varphi^{-1}T\s^3\big(\frac{1}{\sqrt 2}\big) \Big )$ denote the subspace determined by the restrictions to $\varphi(\t)$ of the conformal fields on $\s^3\big(\frac{1}{\sqrt 2}\big )$. These vector fields can be conveniently described as follows. Let $a=(a_1,a_2,a_3,a_4)\in \R^4$. Then the elements of $\mathcal{W}_{{\rm conf}}$ have the following form:
$$
V_{a}=a-2\langle a, \varphi \rangle \varphi \,.
$$
Since $\varphi(\t)$ is not contained in any hyperplane of $\R^4$, it is easy to check that $\dim(\mathcal{W}_{{\rm conf}})=4$. It is well-known that conformal fields have been used to prove the instability of harmonic maps into $\s^n$, $ n \geq3$. In our example, a computation using \eqref{Jacobi-operator} and standard properties of $d$ and $d^{\ast}$ shows that, for all $a \in \R^4, \,a \neq \vec{0}$,
\[
(J(V_a),V_a)=2 \int_{\t} \left(2 [4 \langle a, \varphi \rangle^2-|a|^2]+ \langle V_\gamma,V_a\rangle^2+\langle V_\vartheta,V_a\rangle^2\right) dV_{\t}=-\pi^2 |a|^2
\]
\[
(V_a,V_a)=\int_{\t}\left( |a|^2-2\langle a, \varphi \rangle^2\right)dV_{\t}=\frac{3}{4} \pi^2 |a|^2
\]
and from this 
\[
\frac{(J(V_a),V_a)}{(V_a,V_a)}=-\frac{4}{3}>\mu_1 \,.
\]
We deduce that the Hessian is negatively defined on $\mathcal{W}_{{\rm conf}}$, but we have $\mathcal{W}_{{\rm conf}} \cap\mathcal{W}_{\mu_1}=\{ \vec{0}\}$. 

By contrast, we have verified that the Hessian of $\Phi$ is positively defined on the subspace determined by the conformal vector fields on $\s^4$. 
\end{remark}
\vspace{2mm}

The proof of Theorem\link\ref{Th-varphi} has shown that $J^\varphi$ admits the negative eigenvalue $\mu=4-4\sqrt 2 \in (-4,0)$. Therefore, in this case, the hypothesis of Theorem\link\ref{Th-composition} is not verified. Thus, in this case, the study of the quadratic form \eqref{Hessian-varphi} must be carried out directly. The result of this investigation is:
\begin{proposition}\label{proposition-Hphipi}
Let $H(E_2)_{\Phi} (V,W)$ be the quadratic form defined as in \eqref{Hessian-varphi}.
Then
\[
(I_2^{\Phi,\pi} (V),V) \geq 0 \quad \forall\,V \in \mathcal{W}^{\tilde{\varphi}}\,.
\]
\end{proposition}
\begin{proof}
We know from Theorem\link\ref{Th-Main} that Index$_2(\Phi)=1$ and $I_2^\Phi(V_\eta)=\tilde{\mu} V_\eta$, with $\tilde{\mu}=-16 <0$. We argue by contradiction and assume that $I_2^{\Phi,\pi}$ has a negative eigenvalue. Then there exists $V^* \in \mathcal{W}^\varphi$ such that $I_2^{\Phi,\pi}(V^*)=\mu^* V^*$, with $\mu^* <0$. Because both $V^*$ and $I_2^\Phi(V^*)$ are orthogonal to $V_\eta$, it is easy to conclude that $I_2^\Phi$ would be negatively defined on the $2$-dimensional subspace spanned by $V^*$ and $V_\eta$, a fact which contradicts Theorem\link\ref{Th-Main} and completes the proof of the proposition.
\end{proof}

\begin{remark}
We point out that Proposition~\ref{proposition-Hphipi} can also be verified by using the method of  Theorem\link\ref{Th-Main}.
\end{remark}

\begin{remark}
Obviously, $\big ({\rm Ker}(I_2^{\Phi})\cap \mathcal{W}^\varphi \big ) \subset {\rm Ker}(I_2^{\Phi,\pi})$. Using the method of  Theorem\link\ref{Th-Main}, we could verify that actually $\big ({\rm Ker}(I_2^{\Phi})\cap \mathcal{W}^\varphi \big ) = {\rm Ker}(I_2^{\Phi,\pi})$. More precisely, we observed first that, in a similar fashion to \eqref{eq-general-section-toro-sfera}, each section $V \in \mathcal{W}^\varphi$ can be written as
\begin{equation}\label{eq-general-section-toro-sfera-varphi}
V=f_1 \,V_\gamma +f_2\,V_\vartheta+f_3 \,V_\nu  \,,
\end{equation}
where $f_j\in C^\infty \left ( \t \right)$, $j=1,\ldots,3$. Then, we defined the subspaces $S^0$, $S^{m,n}$ taking into account \eqref{eq-general-section-toro-sfera-varphi}. Next, we found that the only contributions to the nullity are $3$ from $S^0$ and $4$ from $S^{1,1}$, so that $\dim \big ({\rm Ker}(I_2^{\Phi,\pi})\big )=7$.
Finally, we observed that $V_\nu$ and $V_1,\ldots,V_6$ belong to $\mathcal{W}^\varphi$ and so they provide a basis for ${\rm Ker}(I_2^{\Phi,\pi})$. By contrast, $V_7,\ldots,V_{10}$ are not in $\mathcal{W}^\varphi$. 
\end{remark}

We complete our analysis by means of the $p$-harmonic extension of Theorem\link\ref{Th-varphi}, a result which shows how index and nullity may depend on $p$. For simplicity, we shall assume $p \geq 1$.
\begin{theorem}\label{Th-varphi-p-harmonic}
Assume that $p \geq 1$ and let $\varphi:\t \to \s^3(\tfrac{1}{\sqrt 2}) $ be the minimal Clifford torus. Then
\begin{equation*}\label{eq-index-p=4}
\begin{array}{llll}
{\rm (i)}&{\rm Index}_{\rm p-harm}(\varphi)&=& 4 \quad{\rm if}\,\,1 \leq p <4\,; \\
{\rm (ii)}&{\rm Index}_{\rm p-harm}(\varphi)&=& 0 \quad{\rm if}\,\,p \geq 4\,.
\end{array}
\end{equation*}
\begin{equation*}\label{eq-nullity-p<4}
\begin{array}{lllll}
{\rm (i)}&{\rm Null}_{\rm p-harm}(\varphi)&=& 7\quad&{\rm if}\,\,1 \leq p <4\,;  \\ 
{\rm (ii)}&{\rm Null}_{\rm p-harm}(\varphi)&=& 11 \quad&{\rm if}\,\,p=4\,; \\ 
{\rm (iii)}&{\rm Null}_{\rm p-harm}(\varphi)&=& 7 \quad&{\rm if}\,\,p>4\,.  
\end{array}
\end{equation*}
\end{theorem}
\begin{proof} The proof follows exactly the lines given for Theorem\link\ref{Th-varphi} and so we omit the details. To the benefit of  the interested reader, we just point out that, in this context, the version of Proposition\link\ref{proposizione-J-esplicito-toro-sfera} is:
\begin{proposition}\label{proposizione-Jp-esplicito-toro-sfera} Let $f\in C^{\infty}\left ( \t \right )$ be an eigenfunction of $\Delta$ with eigenvalue $\lambda$. Then
\begin{equation}\label{eq-Jp-fV-tutti}
\begin{array}{llll}
{\rm (i)}&J_p \left ( f\, V_\gamma \right )&=&-(p-2)2^{\frac{p-4}{2}}\big(4 f_{\gamma \gamma}V_\gamma+4 f_{\gamma \vartheta}V_\vartheta \big )+2^{\frac{p-2}{2}}\,J (f \,V_\gamma) \\
&&&\\
{\rm (ii)}&J_p \left ( f\, V_\vartheta \right )&=&-(p-2)2^{\frac{p-4}{2}}\big(4 f_{\gamma \vartheta}V_\gamma+4 f_{\vartheta \vartheta}V_\vartheta \big )+2^{\frac{p-2}{2}}\,J (f \,V_\vartheta) \\
&&&\\
{\rm (iii)}&J_p \left ( f\, V_\nu \right )&=& 2^{\frac{p-2}{2}}\,J (f \,V_\nu)\,,
\end{array}
\end{equation}
where $J(V)$ is given in Proposition\link\ref{proposizione-J-esplicito-toro-sfera}.
\end{proposition}
The calculations required to prove this proposition just amount to use \eqref{Jp-relation-J} with $m=2$ and compute explicitly the $d^*$ term.
\end{proof}
\section{Equivariant Index and Nullity}
In Theorems\link\ref{Th-Main} and \ref{Th-Killing} we obtained a complete description of the second variation of $\Phi:\t \to \s^4$. From these results it appears that there exist essentially two geometrically significant directions, that is $V_\eta$, which determines the index, and $V_\nu$, which is the only direction of ${\rm Ker}I_2$ which is not generated by a Killing field either of the ambient or of the domain. 

The Clifford torus $\Phi(\t)$ in $\s^4$ is a $G$-invariant submanifold, with $G={\rm SO}(2) \times {\rm SO}(2)$. The main aim of this section is to show that the key directions $V_\eta,\,V_\nu$ could also be determined by a direct analysis in the orbit space $\s^4 \slash G$. 

It is convenient to consider the following family of maps:
\begin{eqnarray}\label{eq-appendix-mapdatoroasfera} \nonumber
\Phi_{\eta,\nu}:\s^1(R_1)\times \s^1(R_2)  &\to &\s^4 \subset \R^2 \times\R^2 \times \R\\
    \left ( \gamma, \vartheta \right )  &\mapsto &\Big ((\sin \eta \sin \nu )e^{i \gamma},( \sin \eta \cos \nu) e^{i \vartheta}, \cos \eta   \Big ) \,, \; 0 \leq \gamma, \vartheta \leq 2 \pi \,.
\end{eqnarray}
In a map of the type \eqref{eq-appendix-mapdatoroasfera}, we assume that $R_1,R_2>0$ are fixed. 

A map of the type \eqref{eq-appendix-mapdatoroasfera} is $G$-equivariant, i.e., $\Phi_{\eta,\nu}(g x)=g\Phi_{\eta,\nu}( x)$ for all $x \in \t, \,g \in G$ (here $G$ acts naturally on $\t$, and on the first $4$ coordinates of $\R^4$ concerning the target). 

In fact, the family of all the maps of the type \eqref{eq-appendix-mapdatoroasfera} form the set $\Sigma$ of the symmetric points of $C^{\infty}(\t,\s^4)$ with respect to the action of $G$ (see \cite{Palais}). In our case, $\Sigma$ is $2$-dimensional and the tangent space to $\Sigma$ at a symmetric point $\Phi_{\eta,\nu}$ is generated by
\[
\left \{\frac{\partial \Phi_{\eta,\nu}}{\partial \eta}\,, \frac{\partial \Phi_{\eta,\nu}}{\partial \nu} \right \}\,.
\]
The orbit space $\t \slash G$ is just a point, while the orbit space $\mathcal{Q}=\s^4 \slash G$ is naturally identified with the spherical sector $0\leq \eta \leq \pi$, $0 \leq \nu \leq \pi \slash 2$.

By way of summary, at the level of orbit spaces, the map $\Phi_{\eta,\nu}$ can be identified with the point $(\eta,\nu) \in \mathcal{Q}$. To proceed  further, we need to compute the bienergy function.

To this purpose, we observe that the tension field of a map of type \eqref{eq-appendix-mapdatoroasfera} can be computed using
\[
\tau (\Phi_{\eta,\nu})=- \Delta \Phi_{\eta,\nu} + \big | d \Phi_{\eta,\nu} \big |^2 \Phi_{\eta,\nu} 
\]
and, writing $M$ for $\s^1(R_1)\times \s^1(R_2)$, that leads us to
\begin{equation*}\label{bienergy-Phietanu}
E_2(\Phi_{\eta,\nu})=\frac{1}{2}\int_{M}\big | \tau (\Phi_{\eta,\nu})\big |^2 \,dV_{M}=\int_M \,\hat{E}_2(\eta,\nu)\,dV_M\,,
\end{equation*}
where, setting $c=1\slash (32 R_1^4 R_2^4)$,
\begin{eqnarray*}
\hat{E}_2(\eta,\nu)&=&c \Big \{\big(5 R_1^4 - 2 R_1^2 R_2^2 + 
    5 R_2^4 + (3 R_1^4 + 2 R_1^2 R_2^2 + 3 R_2^4) \cos(
      2 \eta)\big ) \sin^2\eta \\
   && - 
 2 (R_1^2 - R_2^2)^2 \cos(4 \nu) \sin^4 \eta + 
 2 (R_1^4 - R_2^4) \cos(2 \nu) \sin^2(2 \eta)\Big \}\,.
\end{eqnarray*}
We shall call $\hat{E}_2(\eta,\nu)$ the \textit{reduced bienergy function}. We point out that, alternatively, the explicit expression of $\tau(\Phi_{\eta,\nu})$ can also be obtained using the reduced energy function and the fact that $\tau(\Phi_{\eta,\nu})$ is tangent to $\Sigma$ at $\Phi_{\eta,\nu}$ (see \cite{MOR4}). 

Now, we are in the cohomogeneity zero case of the reduction theory introduced in \cite{Hsiang}. Then, according to \cite[Proposition~2.5]{BMOR1}, the map $\Phi_{\eta,\nu}$ is biharmonic if and only if $(\eta, \nu)$ is a critical point of $\hat{E}_2$, that is 
\begin{equation}\label{reduced-biharmonicity-equation}
\frac{\partial \hat{E}_2}{\partial \eta}(\eta,\nu)=0 \quad {\rm and}\quad
\frac{\partial \hat{E}_2}{\partial \nu}(\eta,\nu)=0 \,.
\end{equation}
Moreover, the map $\Phi_{\eta,\nu}$ is an isometric immersion if and only if
\begin{equation}\label{cond-isome-immers}
R_1=\sin \eta \sin \nu \quad {\rm and} \quad R_2= \sin \eta \cos \nu \,.
\end{equation}
As we are looking for proper biharmonic immersions which are not congruent, we can assume that $0<\eta , \nu < \pi \slash 2$. Then we check that the only possibility to satisfy both \eqref{reduced-biharmonicity-equation} and \eqref{cond-isome-immers} is
\[
\eta^*= \frac{\pi}{4}\,, \quad \nu^*= \frac{\pi}{4}
\]
which give $R_1=R_2=1 \slash 2$ and so $\Phi_{\eta^*,\nu^*}=\Phi$. The counterparts of $V_\eta,\,V_\nu$ in the orbit space can be described as follows. Let 
\[
\pi :\s^4 \to \big(\mathcal{Q}, \sin^2 \eta \,d \nu^2 + d \eta^2 \big )
\]
be the canonical projection. Then $V_\eta,\,V_\nu$ are horizontal with respect to $\pi$ and 
\begin{equation}\label{projectionVetaVnu}
d\pi(V_\eta)= \frac{\partial}{\partial \eta}\,,\quad d\pi(V_\nu)= \sqrt 2\,\frac{\partial}{\partial \nu}\,.
\end{equation}
Now, we can proceed to the study of the \textit{equivariant second variation}, a notion which was introduced in \cite{MOR3}.  

To this purpose, we compute the Hessian matrix of $\hat{E}_2$ at the critical point $(\eta^*,\nu^*)$. Setting $R_1=R_2=(1\slash 2)$, we obtain
\[
\left [
\begin{array}{rr}
\displaystyle{\frac{\partial^2 \hat{E}_2}{\partial \eta^2}}&\displaystyle{\frac{\partial^2 \hat{E}_2}{\partial \eta \partial \nu} }\\
&\\
\displaystyle{\frac{\partial^2 \hat{E}_2}{\partial \eta \partial \nu}} &\displaystyle{\frac{\partial^2 \hat{E}_2}{\partial \nu^2}}
\end{array}
 \right ]_{(\eta^*,\nu^*)}=
\left [
\begin{array}{rr}
-16&0 \\
0&0
\end{array}
 \right ]\left (=\frac{1}{{\rm Vol}(M)}\left [
\begin{array}{rr}
( I_2(V_\eta),V_\eta )&\frac{1}{\sqrt 2}( I_2(V_\nu),V_\eta ) \\
&\\
\frac{1}{\sqrt 2} ( I_2(V_\eta),V_\nu )&\frac{1}{ 2} (I_2(V_\nu),V_\nu )
\end{array}
 \right ] \right )\,.
\]
Therefore, this analysis in the orbit space at the critical point $(\eta^*,\nu^*)$ tells us that the \textit{equivariant index} of $\Phi$ is equal to $1$ (eigenvalue $\mu_1=-16$, unit eigenvector $\partial \slash \partial \eta$). Moreover, also the \textit{equivariant nullity} of $\Phi$ is equal to $1$ (unit eigenvector $\sqrt 2\,\partial \slash \partial \nu$). These results, together with \eqref{projectionVetaVnu}, suggest that in this example the orbit space analysis displays all the significant second variation features of $\Phi$.

\begin{remark} The method of this section can be extended to other examples. For instance, one could apply these arguments to the ${\rm SO}(\ell+1) \times {\rm SO}(\ell+1)$-invariant proper biharmonic immersions $\Phi_\ell:\s^\ell\big (\frac{1}{2} \big ) \times \s^\ell\big (\frac{1}{2} \big ) \to \s^{2\ell+2}$, $\ell \geq 2$. Of course, it would be nice to be able to exclude,  when $\ell \geq2$, that there exist directions different from $V_\eta,\,V_\nu$ which are geometrically significant for the study of the second variation of $\Phi_\ell$ (index or nullity).   
\end{remark}

\begin{remark} In the case of Theorem\link\ref{Th-varphi} the submanifold $\varphi(\t)$ is again ${\rm SO}(2) \times {\rm SO}(2)$-invariant and the orbit space is $1$-dimensional. In this example, there are $4$ geometrically significant directions which determine the index (see \eqref{base-autovett-varphi}), but it is not possible to recover this $4$-dimensional subspace by carrying out a simplified analysis in the orbit space.
\end{remark}

\begin{remark}
We point out that, in our example $\Phi:\t \to \s^4$, \textit{the normal bundle has dimension equal to two} and so it is interesting to compare this situation with the recent results proved by Ou  concerning the normal stability of certain proper biharmonic \textit{hypersurfaces} of the Euclidean sphere (see \cite{Ou1}). 
\end{remark}

\end{document}